\documentclass[12pt]{article} %%% Version from January 21, 2010 extended September 30, 2016 slightly revised November 27, 2016, further small revision January 5, 2022, final revision November 27, 2023, editorial changes February 18, 2024

\usepackage{amsmath,amssymb,amsthm}
\usepackage[textwidth=160mm,textheight=220mm]{geometry}
\usepackage[cmtip,all]{xypic}

\let\smash=\wedge
\let\iso=\cong

\let\tensor=\otimes
\let\minus=\smallsetminus

\newcommand{\colim}{\operatornamewithlimits{colim}}
\newcommand{\hofib}{\operatornamewithlimits{hofib}}
\newcommand{\Hom}{\operatorname{Hom}}
\newcommand{\inthom}{\underline{\operatorname{Hom}}}

\newcommand{\Fr}{\operatorname{Fr}}

\newcommand{\Sing}{\operatorname{Sing}}

\newcommand{\Aff}{\mathbb{A}}
\newcommand{\PP}{\mathbb{P}}
\newcommand{\ZZ}{\mathbb{Z}}
\newcommand{\FF}{\mathbb{F}}

\newcommand{\NN}{\mathbb{N}}
\newcommand{\QQ}{\mathbb{Q}}

\newcommand{\CC}{\mathbb{C}}
\newcommand{\LL}{\mathbb{L}}
\newcommand{\one}{\mathbb{I}}

\newcommand{\op}{\mathrm{op}}

\newcommand{\id}{\mathrm{id}}
\newcommand{\Id}{\mathrm{Id}}
\newcommand{\Sm}{\mathrm{Sm}}

\newcommand{\SH}{\mathrm{SH}}

\newcommand{\Var}{\mathrm{Var}}

\newcommand{\sym}{\mathrm{sym}}

\newcommand{\Sdot}{\mathcal{S}_{\bullet}}

\newcommand{\wC}{\mathrm{w}\mathbf{C}}
\newcommand{\cC}{\mathrm{cof}\mathbf{C}}
\newcommand{\wM}{\mathrm{w}\mathbf{M}}
\newcommand{\cM}{\mathrm{cof}\mathbf{M}}

\newcommand{\fib}{\mathrm{fib}}
\newcommand{\cof}{\mathrm{cof}}

\newcommand{\Set}{\mathbf{Set}}
\newcommand{\Top}{\mathbf{Top}}
\newcommand{\sSet}{\mathbf{sSet}}
\newcommand{\M}{\mathbf{M}}
\newcommand{\N}{\mathbf{N}}
\newcommand{\R}{\mathbf{R}}
\newcommand{\Sp}{\mathbf{Spt}}
\newcommand{\SymSp}{\mathbf{SymSpt}}

\newcommand{\f}{\mathrm{fin}}
\newcommand{\g}{\mathrm{g}}
\newcommand{\du}{\mathrm{dual}}
\newcommand{\hf}{\mathrm{hfin}}
\newcommand{\ifin}{\mathrm{ifin}}
\newcommand{\ihf}{\mathrm{ihfin}}

\newcommand{\E}{\mathsf{E}}
\newcommand{\D}{\mathsf{D}}
\newcommand{\B}{\mathsf{B}}

\newcommand{\AAA}{\mathrm{A}}

\newcommand{\Tan}[1]{\mathcal{T}{({#1})}}
\newcommand{\Th}[1]{\mathrm{Th}{({#1})}}
\newcommand{\Spec}[1]{\mathrm{Spec}{({#1})}}

\theoremstyle{plain}
\newtheorem{theorem}[subsection]{Theorem}
\newtheorem{theorem*}{Theorem}
\newtheorem{lemma}[subsection]{Lemma}
\newtheorem{corollary}[subsection]{Corollary}
\newtheorem{proposition}[subsection]{Proposition}
\theoremstyle{definition}
\newtheorem{remark}[subsection]{Remark}
\newtheorem{defn}[subsection]{Definition}
\newtheorem{example}[subsection]{Example}

\begin{document}

\title{The Grothendieck ring of varieties and algebraic $K$-theory of spaces}
\author{Oliver R{\"o}ndigs\footnote{Institut f\"ur Mathematik, Universit\"at Osnabr\"uck, Germany}\, \footnote{Matematisk Institutt, Universitetet i Oslo, Norway}}
\date{Revised November 25, 2023}

\maketitle

\begin{abstract}
  Waldhausen's algebraic $K$-theory machinery is applied to Morel-Voevodsky
  $\mathbb{A}^1$-homotopy, producing an interesting 
  $\mathbb{A}^1$-homotopy type. Over a field $F$ of
  characteristic zero, its path components receive a 
  surjective ring homomorphism
  from the Grothendieck ring of varieties over $F$.
\end{abstract}

\section{Introduction}
\label{sec:introduction}

Waldhausen's approach to algebraic $K$-theory \cite{Waldhausen}
is of such a generality
(though not for its own sake) that it applies to a wide class of homotopy
theories. The choice here, as suggested by Waldhausen already in the
last century, is the $\Aff^1$-homotopy theory over a
Noetherian finite-dimensional base scheme $S$ introduced by 
Morel and Voevodsky \cite{mv}. Subject to an appropriate
finiteness condition (for which there are several choices),
the resulting homotopy type $\AAA(S)$
is nontrivial;
for example, it contains Waldhausen's algebraic $K$-theory of a point,
$\AAA(\ast)$, as a retract up to homotopy. Moreover, it can be viewed
as an $\Aff^1$-homotopy type in a natural way. The present paper is an
admittedly rather meagre attempt to advertise this $\Aff^1$-homotopy
type to algebraic geometers, although it might be more attractive to
homotopy theorists. Recall that almost by construction,
the path components of Waldhausen's $K$-theory provide
the universal Euler characteristic.

\begin{theorem*}
  Let $F$ be a field of characteristic zero. Sending a smooth
  projective variety to its natural class in $\pi_0\AAA(F)$
  defines a surjective ring homomorphism
  \[ K_0(\Var_F)\rightarrow \pi_0\AAA(F) \]
  from the Grothendieck ring of varieties over $F$.
\end{theorem*}

See Theorem~\ref{thm:a-motivic-measure} for a precise version including the
appropriate finiteness condition.
This ring homomorphism refines several other motivic measures,
such as the topological Euler characteristic, the Hodge motivic measure,
and the Gillet-Soul\'e motivic measure. In these cases the ring
homomorphism is naturally induced on path
components by a map from the homotopy type $\AAA(F)$,
thereby solving at least \cite[Problems 7.3, 7.4, 7.5]{cwz} in a
natural way.
The constructions
\cite{campbell, zakharevich.perspectives} supply a homotopy type whose
path components is the Grothendieck ring of varieties over $F$.
However, the significance of its higher homotopy groups is not clear.
In the case of Waldhausen's original application of algebraic $K$-theory
to the geometry of manifolds, the higher homotopy groups
yield interesting information on their automorphism 
groups \cite{farrell-hsiang, rognes.wheven, rognes.whodd},
thanks to the following statement from \cite{waldhausen.concordance}.

\begin{theorem*}[Waldhausen]\label{thm:waldhausen}
  Let $M$ be a smooth manifold with possibly empty boundary.
  The homotopy type $\AAA(M)$ is defined as the Waldhausen $K$-theory of
  the category of finite cell complexes retractive over $M$.
  There is a splitting  
  \[\AAA(M) \simeq \Sigma^\infty M_+ \times \mathrm{Wh}(M) \]
  up to homotopy,
  where $\Sigma^\infty M_+$ is the stable homotopy type of $M$,
  and $\mathrm{Wh}(M)$ is the Whitehead spectrum of $M$, a double
  delooping of the spectrum of stable smooth pseudoisotopies of $M$.
\end{theorem*}

Taking path components of this splitting recovers the
s-cobordism theorem of Smale, Barden, Mazur, and Stallings.
Perhaps an $\Aff^1$-s-cobordism theorem for smooth varieties
over a field can be produced via the trace given in Section~\ref{sec:trace}
on the $\Aff^1$-homotopy type obtained from the algebraic $K$-theory
of dualizable motivic $T$-spectra.

\section{$K$-theory of model categories}
\label{sec:k-theory-model}

In \cite{Waldhausen}, Waldhausen generalized Quillen's $K$-theory machinery
to the setup of categories with cofibrations and weak equivalences,
henceforth called {\em Waldhausen
categories\/}. 
A Waldhausen category is a quadruple 
$(\mathbf{C},\ast,\wC,\cC)$, where $\mathbf{C}$ is
a pointed category with zero object $\ast$, a subcategory $\wC$ of
weak equivalences and a subcategory $\cC$ of cofibrations. Furthermore, 
$\ast \rightarrow A$ is always a cofibration, cobase changes
along cofibrations exist in $\mathbf{C}$, and the weak equivalences satisfy
the gluing lemma. The {\em algebraic\/} $K$-{\em theory\/}
of a Waldhausen category $(\mathbf{C},\ast,\wC,\cC)$ is the
spectrum
\begin{equation}\label{eq:a-general} 
  \AAA(\mathbf{C}) = (\wC,\mathrm{w}\Sdot \mathbf{C},\dotsc,\mathrm{w}\Sdot^{(n)}\mathbf{C},\dotsc )
\end{equation}
of pointed simplicial sets obtained by the diagonal of the nerve
of $n$-fold simplicial categories; the latter produced by iterated
applications of Waldhausen's $\Sdot$-construction. The structure maps
of~(\ref{eq:a-general}) are induced by the inclusion of the 1-skeleton.

\begin{defn}\label{def:a-equivalence}
  An exact functor $F\colon \mathbf{C}\rightarrow \mathbf{D}$ of
  Waldhausen categories is a $K$-{\em theory equivalence\/}
  if the induced map $\AAA(F)\colon \AAA(\mathbf{C})\rightarrow \AAA(\mathbf{D})$
  of spectra is a stable equivalence.
\end{defn}

All the categories with cofibrations and weak equivalences
in the following are obtained as full subcategories of a 
Quillen model category,
where the weak equivalences are determined by the model structure.
The cofibrations are either determined by the model structure, or by 
a slight variation.
References for model categories are \cite{Hovey:book, Hirschhorn}.
Algebraic $K$-theory
requires finiteness conditions, and suitable cofibrantly
generated model categories, as defined
in \cite[Definition 4.1]{Hovey:stable-model}
and \cite[Definition 3.4]{dro.enriched}, provide a convenient
setup for these.

\begin{defn}\label{defn:weakly-fin-gen}
  A model category $\M$ is {\em weakly finitely generated\/}
  if it is cofibrantly generated and satisfies the 
  following further requirements:
  \begin{enumerate}
    \item There exists a set $I$ of generating cofibrations
      with finitely presentable domains and codomains.
    \item There exists a set $J$ of acyclic cofibrations
      with finitely presentable domains and codomains detecting
      fibrations with fibrant codomain.
  \end{enumerate}
\end{defn}

Examples of weakly finitely generated model categories are
the usual model categories of (pointed) simplicial sets
(denoted $\sSet_\bullet$),
spectra of such (denoted $\Sp$), chain complexes over a ring, and suitable
model structures for $\Aff^1$-homotopy theory.
The latter is essentially
a consequence of the following statement.

\begin{proposition}\label{prop:weakly-fin-gen-inherit}
  Let $\M$ be a weakly finitely generated simplicial 
  model category, 
  and let $Z$ be a set of morphisms in $\M$ with finitely
  presentable domains and codomains. Suppose that
  tensoring with a finite simplicial set $L$ preserves
  finitely presentable objects.
  If the left Bousfield localization $\mathrm{L}_{Z} \M$ 
  exists, it is weakly finitely generated.
\end{proposition}

\begin{proof}
  The proof of \cite[Proposition 4.2]{Hovey:stable-model}
  applies.
\end{proof}

In a weakly finitely generated model category $\M$, a fibrant
replacement functor \[ \fib\colon \M \rightarrow \M, \] 
which commutes with filtered colimits, can be constructed 
by attaching cells from a
set of acyclic cofibrations
$J$ with finitely presentable domains and codomains. 
It follows that the natural transformation $\Id_\M\rightarrow \fib$ is
an acyclic cofibration.

For applications in algebraic $K$-theory recall that,
given an object $B\in \M$ in a  category, an object 
{\em retractive over\/} $B$ is a pair
$(B\xrightarrow{s} D,D\xrightarrow{r}B)$ of morphisms in $\M$ such that $r\circ s=\id_B$.
Such a pair will often be abbreviated as ``$D$''.
With the obvious notion of morphism, these form a
category $\R(\M,B)$. A morphism $\phi\colon B\rightarrow C$ in $\M$
induces a functor
\[\phi_! \colon \R(\M,B) \rightarrow \R(\M,C), \quad D  \mapsto  D\cup_B C \]
having the functor
\[ \phi^! \colon \R(\M,C) \rightarrow \R(\M,B), \quad E \mapsto E\times_C B \]
as right adjoint, provided pushouts and pullbacks exist.

\begin{proposition}\label{prop:retractive-model}
  Let $\M$ be a weakly finitely generated 
  model category, and $B$ an object of $\M$. 
  The category $\mathbf{R}(\M,B)$ of objects retractive over $B$ is a
  weakly finitely generated 
  model category in a natural way. For
  every $\phi\colon B\rightarrow C$, the pair $(\phi_!,\phi^!)$ is Quillen.
  If $\M$ is simplicial, 
  then so is $\mathbf{R}(\M,B)$. If $\M$ is a monoidal
  model category under the cartesian product, then
  $\R(\M,B)$ is a $\R(\M,\ast)$-model category. 
\end{proposition}

\begin{proof}
  The statement regarding the (simplicial) model structure
  is \cite[Prop.~1.2.2]{schwede.spectra}, which implies
  the Quillen pair property. The statement
  regarding the generators can be deduced 
  from \cite[Lemma 1.3.4]{schwede.spectra}. For later reference,
  if $I=\{s_i\hookrightarrow t_i\}_{i\in I}$ is the set of generating cofibrations
  in $\M$, then 
  \[ \{B\coprod (s_i \hookrightarrow t_i)\}_{i\in I, \psi\in \Hom_{\M}(t_i,B)} \]
  is the set of generating cofibration in $\R(\M,B)$, where
  the maps $\psi\colon t_i\rightarrow B$ define the required retractions.
  Note that $\phi_!$ preserves this set of generating cofibrations.
  The final statement follows from \cite[Prop.~4.2.9]{Hovey:book}
  and the standard pairing
  \begin{align*}
    \R(\M,B)\times \R(\M,C) & \rightarrow \R(\M,B\times C), & 
    (D,E)
     \mapsto  D\times E\cup_{(B\times E\cup_{B\times C}
      D\times C)} B\times C,
  \end{align*}
  of retractive objects, which is a Quillen bifunctor.
\end{proof}

\begin{defn}\label{def:b-spectrum}
  Let $\M$ be a symmetric monoidal 
  model category, let $\N$ be an $\M$-model category,
  and let $B$ be an object of $\M$. 
  A $B$-{\em spectrum\/} $\E$ in $\N$  
  consists of a
  sequence $(\E_0,\E_1,\dotsc)$ of objects in $\N$
  together with a sequence of structure maps
  \[ \sigma_n^{\E}\colon \Sigma_B\E_n:=\E_n\smash B \rightarrow \E_{n+1}.\]
  The category of
  $B$-spectra in $\N$ is denoted $\Sp_B(\N)$. 
  Set $\Sp(\N):=\Sp_{S^1}(\N)$, where $S^1\in \sSet_\bullet =\M$ 
  is the category
  of pointed simplicial sets.
\end{defn}

\begin{proposition}\label{prop:stabilize-model-structure}
  Let $\M$ be a symmetric monoidal $\sSet_\bullet$-model category, 
  let $\N$ be an $\M$-model category,
  and let $B$ be a finitely presentable and
  cofibrant object of $\M$.
  Suppose that
  tensoring with a finite simplicial set $L$ preserves
  finitely presentable objects in $\N$. 
  If $\N$ is weakly finitely generated,
  then $\Sp_B(\N)$ is a weakly finitely generated model category
  such that $\Sigma_B^\infty \colon \N \rightarrow \Sp_B(\N)$ is 
  a left Quillen functor.
\end{proposition}

\begin{proof}
  The proof given for \cite[Theorem 4.12]{Hovey:stable-model} applies.
  For later reference, the required sets are listed explicitly.
  The levelwise model structure on $\Sp_B(\N)$
  is weakly finitely generated
  with the sets 
  \[ \Fr I := \{\Fr_n i\}_{n\geq 0,i\in I} \quad \mathrm{and} \quad
  \Fr J :=\{\Fr_n j\}_{n\geq 0,j\in J},\]
  where $\Fr_n$ is the left adjoint
  of the evaluation 
  functor sending $\E$ to $\E_n$, and
  $I$ and $J$ are sets of maps in $\N$ satisfying
  Definition~\ref{defn:weakly-fin-gen}.
  The statement follows from Proposition~\ref{prop:weakly-fin-gen-inherit}
  because the $B$-stable model structure is a left Bousfield
  localization of the levelwise model structure with respect
  to the set
  \[ \{ \Fr_{n+1} (C\smash B) \to \Fr_n C \}_{n\in \mathbb{N},C \mathrm{\ domain\ or\ codomain\ in\ }I} \] 
  of morphisms with finitely presentable domains and codomains.
\end{proof}

As soon as $B$ is a suspension (for example, $S^1$ itself),
the model structure from Proposition~\ref{prop:stabilize-model-structure}
on $B$-spectra is stable in the sense of \cite[Definition 7.1.1]{Hovey:book}.
In particular, the weak equivalences of $B$-spectra then
satisfy Waldhausen's extension axiom \cite[p.327]{Waldhausen}.
However, since $\Sp_B(\M)$ usually does not inherit any monoidality
properties from $\M$, one has to use Jeff Smith's symmetric $B$-spectra
instead. Consider \cite[Theorem 8.11, Corollary 10.4]{Hovey:stable-model} for the following.

\begin{proposition}\label{prop:symm-spectra}
  Let $\M$ be a symmetric monoidal $\sSet_\bullet$-model category,
  let $\N$ be an $\M$-model category,
  and let $B$ be a finitely presentable and
  cofibrant object of $\M$.
  Suppose that
  tensoring with a finite simplicial set $L$ preserves
  finitely presentable objects. 
  If $\N$ is weakly finitely generated,
  then $\SymSp_B(\N)$ is a weakly finitely generated 
  $\SymSp_B(\M)$-model category
  such that $\Sigma_B^\infty \colon \N \rightarrow \SymSp_B(\N)$ is 
  a left Quillen functor. If additionally the 
  cyclic permutation on $B\smash B\smash B$ is homotopic
  to the identity, the model categories
  $\SymSp_B(\N)$ and  $\Sp_B(\N)$ 
  are Quillen equivalent.
\end{proposition}

\begin{proof}
  The proof given for 
  \cite[Theorem 8.11, Corollary 10.4]{Hovey:stable-model} applies.
  For later reference, the required sets are listed explicitly.
  The levelwise model structure on the category $\SymSp_B(\M)$
  is weakly finitely generated
  with the sets 
  \[ \Fr^\sym I := \{\Fr_n^\sym i\}_{n\geq 0,i\in I} \quad \mathrm{and} \quad
  \Fr^\sym J :=\{\Fr_n^\sym j\}_{n\geq 0,j\in J},\]
  where $\Fr_n^\sym$ is the left adjoint
  of the evaluation 
  functor sending $\E$ to $\E_n$, and
  $I$ and $J$ are sets of maps in $\M$ satisfying
  Definition~\ref{defn:weakly-fin-gen}.
  The statement follows from Proposition~\ref{prop:weakly-fin-gen-inherit}
  because the $B$-stable model structure is a left Bousfield
  localization of the levelwise model structure with respect
  to the set
  \[ \{ \Fr_{n+1}^\sym (C\smash B) \to \Fr_n^\sym C \}_{n\in \mathbb{N},C \mathrm{\ domain\ or\ codomain\ in\ }I} \] 
  of morphisms with finitely presentable domains and codomains.
\end{proof}

\begin{lemma}\label{lem:stably-contr}
  Let $\M$ be a symmetric monoidal $\sSet_\bullet$-model category, 
  let $\N$ be an $\M$-model category,
  and let $B$ be a finitely presentable and
  cofibrant object of $\M$.
  Suppose that
  tensoring with a finite simplicial set $L$ preserves
  finitely presentable objects in $\N$. 
  Suppose that $\N$ is weakly finitely generated, and 
  let $\Sp_B(\N)$ be the stable model category of $B$-spectra
  in $\N$. If $\E$ is a cofibrant finitely presentable $B$-spectrum
  in $\N$ which is stably contractible, then there exists a natural
  number $N$ such that $\E_n$ is contractible for every $n\geq N$.
\end{lemma}

\begin{proof}
  Let $\E$ be finitely presentable and cofibrant. Then $\E_n$ 
  is cofibrant and finitely presentable for every $n$. Moreover,
  there exists
  a natural number $M$ such that the structure maps $\sigma_m$
  are isomorphisms for every $m\geq M$. Since the canonical map
  \[ \Fr_M(\E_M)\rightarrow \E \]
  from the shifted suspension spectrum of $\E_M$ to $\E$ is a stable
  equivalence, one may work with $\Fr_M(\E_M)$ directly. Moreover, 
  one may choose $M=0$. Thus $\E_0$ is a cofibrant finitely presentable
  object with the property that 
  $\colim_n \Omega^n \fib(\Sigma^n \E_0)$ is contractible. Here
  $\fib\colon \N\rightarrow \N$ is a fibrant replacement functor.
  Equivalently, the class of the canonical map $\E_0
  \rightarrow \Omega^n\fib(\Sigma^n \E_0)$ becomes the class of the
  constant map in the colimit
  \[ [\E_0,\fib(\E_0)] \rightarrow {[\E_0,\Omega \fib(\Sigma_B \E_0)]}\rightarrow
    \dotsm \rightarrow {[\E_0,\Omega^n\fib(\Sigma^n_B \E_0)]} \rightarrow\dotsm \]
  of sets of pointed homotopy classes of maps. As $\E_0$ and
  $\E_0\otimes \Delta^1$ are finitely presentable, there exists a
  natural number $N$ such that the homotopy class of the canonical
  map $\E_0\rightarrow \Omega^n \fib(\Sigma^n_B \E_0)$ coincides with
  the homotopy class of the constant map. By adjointness, the
  canonical map $\Sigma^n_B \E_0 \rightarrow \fib(\Sigma^n_B \E_0)$ is homotopic
  to the constant map for every $n\geq N$. 
  Thus $\Sigma^n_B \E_0$ is contractible for every $n\geq N$.
\end{proof}

Definition~\ref{defn:weakly-fin-gen} leads to the following
finiteness notions. More variations, such as being finitely
dominated, are possible.

\begin{defn}\label{defn:finiteness} Let $\M$
  be a weakly finitely generated pointed model category, and choose
  a set of generating cofibrations
  $I$ with finitely presentable domains and codomains. 
  Let $B\in \M$.
  \begin{enumerate}
    \item The object $B$ is {\em finite\/} if
      it is cofibrant and finitely presentable; the latter means that
      $\Hom_\M(B,-)$ commutes with filtered colimits.
    \item The object $B$ is {\em homotopy finite\/}
      if it is cofibrant and weakly equivalent to a
      finite object.
    \item The object $B$ is $I$-{\em finite\/} if
      the map $\ast \rightarrow B$ is obtained by attaching
      finitely many maps from $I$.
    \item The object $B$ is $I$-{\em homotopy finite\/}
      if it is cofibrant and weakly equivalent to an
      $I$-finite object.
  \end{enumerate}
  The resulting full subcategories are denoted
  $\M^\f,\M^\hf,\M^{\mathrm{ifin}}$, and $\M^{\mathrm{ihfin}}$, respectively. 
\end{defn}

The category $\M^\ifin$ is essentially small, and so is $\M^\f$, at least if
$\M$ is locally finitely presentable. This will usually not be the
case for $\M^\ihf$ and $\M^\hf$. This set-theoretical issue can be
resolved in several ways, but will be ignored in the present approach,
following \cite[Remark on page 379]{Waldhausen}. 
Homotopy finite objects in a weakly finitely generated pointed model
category $\M$ are compact in the homotopy category of $\M$, as the
proof of \cite[Theorem 7.4.3]{Hovey:book} shows. 
In the case where $\M$ is a symmetric monoidal model category,
another finiteness notion is quite natural and will be used
eventually.

\begin{defn}\label{def:dualizable}
  Let $(\M,\smash,\one)$ be a symmetric monoidal model category.
  A cofibrant object $B$ is {\em dualizable\/}
  if there exists a cofibrant object $C$ and morphisms
  $\phi\colon \one\rightarrow B\smash C$, $\psi\colon C\smash B\rightarrow \one$
  in the homotopy category of $\M$, such that the compositions
  \[ B\xrightarrow{\phi\smash B} B\smash C\smash B 
  \xrightarrow{B\smash \psi}B \quad \mathrm{and} \quad 
  C\xrightarrow{C\smash \phi} C\smash B\smash C 
  \xrightarrow{\psi\smash C}C \]
  are the respective identities. The full subcategory
  of cofibrant and dualizable objects is denoted $\M^\du$.
\end{defn}

The categories introduced in Definitions~\ref{defn:finiteness}
and~\ref{def:dualizable} are equipped with a subcategory of weak equivalences
by intersecting with $\wM$, and with a subcategory of
cofibrations by intersecting with $\cM$ in the cases of
$\M^\f$, $\M^\hf$, and $\M^\du$. In the cases of 
$\M^{\mathrm{ifin}}$ and $\M^{\mathrm{ihfin}}$
this may lead to trouble with the required existence of
cobase changes. The subcategory of cofibrations in $\M^{\mathrm{ifin}}$ 
consists of those maps obtained by attaching finitely many cells from $I$,
and in $\M^{\mathrm{ihfin}}$ it is simply maps obtained 
by attaching cells from $I$.

\begin{lemma}\label{lem:fin-ifin}
  Let $\M$
  be a weakly finitely generated model category, and choose
  a set of generating cofibrations
  $I$ with finitely presentable domains and codomains. 
  If $B\hookrightarrow E$ is a cofibration of finitely presentable 
  objects in $\M$, there exists a finite
  $I$-cofibration $B\hookrightarrow C$ in $\M$ such that $B\hookrightarrow E$ is a retract
  of $B\hookrightarrow C$.
\end{lemma}

\begin{proof}
  Factor $B\hookrightarrow E$ via the small
  object argument applied to $I$, to obtain a lifting problem
  \[\xymatrix{
    B \ar[r]^-i \ar[d] & D \ar[d]^-\sim\\
    E \ar[r]^-{\id} &  E}\]
  which can be solved. The object $D$ is a sequential colimit
  of a diagram 
  \[ B =D_{-1} \hookrightarrow D_0 \hookrightarrow \dotsm \hookrightarrow D_n\hookrightarrow D_{n+1}
  \hookrightarrow \dotsm \] such that $D_{n+1}$ is obtained by attaching
  $I$-cells to $D_n$ indexed by a specific subset of the disjoint union
  $\coprod_{i\in I} \Hom_\M(\mathrm{dom}(i),D_n)$ for every $n$. 
  Since $E$ is finitely presentable,
  a lift $E \rightarrow D$ factors via a morphism $E\rightarrow D_{n+1}$. This object
  is the filtered colimit of objects $D_{n,\alpha}$ which 
  are obtained by attaching finitely many cells to $D_n$. This
  colimit is indexed over certain finite subsets 
  $\alpha \subset \coprod_{i\in I} \Hom_\M(\mathrm{dom}(i),D_n)$. 
  Again since $E$ is finitely presentable, there exists a finite
  subset
  $\beta \subset \coprod_{i\in I} \Hom_\M(\mathrm{dom}(i),D_n)$
  and a factorization over $E\rightarrow D_{n,\beta}$.
  Since the domains of the morphisms in $I$ are finitely presentable,
  one may proceed in the same fashion for every domain in $\beta$ 
  inductively to obtain a factorization $E\rightarrow C$ where $C$ is 
  obtained by attaching finitely many $I$-cells.
\end{proof}

Already algebraic examples such as chain complexes $\mathbf{Ch}(R)$ over
a ring $R$  show that the classes of
finite and $I$-finite objects can be different,
also on the level of algebraic $K$-theory. Bounded chain complexes of
finitely generated projective $R$-modules
are the finite objects in the standard model category of all chain complexes
of $R$-modules given in \cite[Def.~2.3.3]{Hovey:book}, whereas the
$I$-finite objects are the bounded chain complexes of finitely generated
free $R$-modules. One may use
\cite[Corollary~II.26.3 and Theorem II.9.2.2]{weibel.kbook}
to conclude that $K_0(\mathbf{Ch}^\ifin(\mathbb{Z}[\sqrt{-5}]))\cong \mathbb{Z}$
and $K_0(\mathbf{Ch}^\f(\mathbb{Z}[\sqrt{-5}]))\cong \mathbb{Z}\oplus \mathbb{Z}/2\mathbb{Z}$.
 
\begin{proposition}\label{prop:finiteness}
  Let $\M$ be a weakly finitely generated model category, and
  choose a set $I$ of generating cofibrations
  with finitely presentable domains and codomains. 
  Then the following hold.
  \begin{enumerate}
    \item\label{item:1} With these choices,
      $\M^\f,\M^\hf,\M^{\mathrm{ifin}}$ and $\M^{\mathrm{ihfin}}$
      are categories with cofibrations and weak equivalences.
    \item\label{item:2} The horizontal functors in the commutative
      diagram
      \begin{equation}\label{eq:2}
	\xymatrix{
	  \M^{\mathrm{ifin}} \ar[r] \ar[d] & \M^{\mathrm{ihfin}}\ar[d]  \\
	  \M^\f \ar[r] & \M^{\hf}
	}
      \end{equation}
      of exact inclusion functors are $K$-theory
      equivalences. 
    \item\label{item:4} The path components of
      the homotopy fiber of the map $\AAA(\M^\mathrm{if} \hookrightarrow \M^\f)$
      are all contractible.
  \end{enumerate}
\end{proposition}

\begin{proof}
  Consider statement~\ref{item:1} first.
  The gluing lemma follows from the cube lemma for model categories
  \cite[Lemma 5.2.6]{Hovey:book}. It remains to check in each case that
  a cobase change in $\M$ along the
  cofibration in question does not lead outside of the 
  category in question. For $\M^\f$ this follows, since 
  a pushout of finitely presentable objects is again finitely
  presentable, and cofibrancy is preserved. 

  In the case of $\M^\hf$, let
  $C \leftarrow B \hookrightarrow D$ be a diagram in $\M^\hf$ such
  that $B\hookrightarrow D$ is a cofibration. Let $B\rightarrow \fib(B)$
  be a fibrant replacement obtained by attaching cells from a
  set $J$ with finitely presentable domains and codomains.
  Choose a weak equivalence $B^\prime \xrightarrow{\sim} \fib(B)$
  from a finite object. The gluing lemma then implies that one
  may choose $B$ to be finite. Similarly to the above, choose a
  finite object $D^\prime$ and weak equivalences 
  $D \xrightarrow{\sim} \fib(D^\prime) \xleftarrow{\sim} D^\prime$. Note that
  $D^\prime \xrightarrow{\sim} \fib(D^\prime)$ is the 
  filtered colimit of certain maps $D^\prime \xrightarrow{\sim} D^{\prime\prime}$
  which are obtained by attaching finitely many maps from $J$.
  Since $B$ is finitely presentable, $B\hookrightarrow D \xrightarrow{\sim} \fib(D^\prime)$
  lifts to such a $D^{\prime\prime}$, and analogously for $C$.
  By assumption on $J$ and the gluing lemma, statement~\ref{item:1} follows.

  Statement~\ref{item:1}
  follows for $\M^{\mathrm{ifin}}$ basically by definition. 
  The argument in the case of $\M^{\mathrm{ihfin}}$
  is similar to the argument in the case of $\M^\hf$, which
  concludes the proof of statement~\ref{item:1}.

  Diagram~(\ref{eq:2}) exists because the domains and codomains of
  the maps in $I$ are finitely presentable. The
  statements for $\AAA(\M^\f)  \rightarrow \AAA( \M^{\hf})$
  and for $\AAA(\M^{\ifin})  \rightarrow \AAA( \M^{\ihf})$ follow
  from \cite[Theorem 2.8]{Sagave}. For statement~\ref{item:4} observe that
  every object in $\M^\f$ is a retract of an object in $\M^{\mathrm{ifin}}$
  by Lemma~\ref{lem:fin-ifin}.
  One then concludes with \cite[Theorem 1.10.1]{thomason.festschrift}.
\end{proof}

\begin{proposition}\label{prop:dualizable}
  Let $\M$ be a symmetric monoidal stable model category.
  Then $\M^\du$ is a Waldhausen category in the natural
  way described above. 
\end{proposition}

\begin{proof}
  It remains to prove that if $C\leftarrow B\hookrightarrow D$ is a diagram
  in $\M^\du$, then its pushout $C\cup_BD$ is again dualizable.
  This follows from \cite[Theorem 0.1]{May:traces}, for example.
\end{proof}

The next statement is quite useful, since it implies
that through the eyes of algebraic $K$-theory, restriction
to stable model categories is acceptable. This in turn
allows the full applicability of Waldhausen's theory, since
the weak equivalences then satisfy the extension axiom.
Its proof goes back
to \cite{waldhausen.chromatic} which led to \cite{roendigs.diplom}.

\begin{theorem}\label{thm:stabilization-A}
  Let $\M$ be a pointed simplicial model category such that
  tensoring with a finite simplicial set preserves finitely presentable objects.
  Suppose that $\M$ is weakly finitely generated, and 
  let $\Sp(\M)$ be the stable model category of $S^1$-spectra
  in $\M$. 
  Then the suspension spectrum functor
  induces a $K$-theory equivalence
  \[ \AAA(\M^\g) \rightarrow \AAA(\Sp^\g(\M)), \]
  where $\g\in \{\f,\hf\}$. If $\Sigma$ preserves cofibrations
  in $\M^{\ifin}$, then the same is true for $\g\in \{\ifin,\ihf\}$.
\end{theorem}

\begin{proof}
  Consider first $\g=\f$.
  By the additivity
  theorem \cite[Theorem 1.4.2]{Waldhausen} the suspension functor
  $\Sigma=-\smash S^1$
  induces a $K$-theory equivalence. Consider the colimit
  of 
  \[ \M^\f \xrightarrow{\Sigma} \M^\f \xrightarrow{\Sigma} \dotsm \]
  in the category of Waldhausen categories. There is an isomorphism
  \[ \colim_\Sigma \Sdot \M^\f \iso \Sdot \colim_\Sigma \M^\f, \]
  which implies that the canonical functor $\M^\f \rightarrow \colim_\Sigma \M^\f$
  is a $K$-theory equivalence. A $S^1$-spectrum $\E$ is 
  {\em strictly finite\/} if there exists a natural number $N = N(\E)$
  such that $\E_N$ is finite and for every $n\geq N$ the structure map 
  $\sigma_n\colon \Sigma \E_n\rightarrow \E_{n+1}$ is the identity.
  Let $\Sp^{\mathrm{sf}}(\M)$ denote
  the full subcategory of strictly finite $S^1$-spectra which are
  also cofibrant. It is a Waldhausen category in a natural way, 
  and the inclusion
  $\Sp^{\mathrm{sf}}(\M)\hookrightarrow \Sp^\f(\M)$ is an exact equivalence. In particular,
  the inclusion is a $K$-theory equivalence.
  
  Let $\Phi\colon \Sp^{\mathrm{sf}}(\M)\rightarrow \colim_\Sigma\M^f$ denote the functor
  sending $\E$ to the equivalence class of $(\E_n,n)$, where $n\geq N(\E)$.
  This functor is well-defined, preserves cofibrations, and pushouts
  essentially by construction.
  Moreover, it preserves weak equivalences by Lemma~\ref{lem:stably-contr}.
  It is straightforward to verify that $\Phi$ satisfies the conditions
  of Waldhausen's Approximation Theorem \cite[Theorem 1.6.7]{Waldhausen}.
  Thus $\Phi$ is a $K$-theory equivalence. It remains to note that
  the suspension spectrum functor $\Sigma^\infty \colon \M^\f
  \rightarrow \Sp^\f(\M)$ factors as
  \[ \M^\f  \rightarrow   \Sp^{\mathrm{sf}}(\M) \hookrightarrow \Sp^\f(\M), \]
  which completes the proof for $\g=\f$. The case $\g=\hf$
  then follows from Proposition~\ref{prop:finiteness}.
  The extra assumption on $\Sigma$ implies that these arguments
  apply also to $\g\in \{\ifin,\ihf\}$.
\end{proof}

Theorem~\ref{thm:stabilization-A} provides many examples of
non-equivalent homotopy theories having the same $K$-theory.

\section{$\Aff^1$-homotopy theory}
\label{sec:aff1-homotopy-theory}

Motivic or $\Aff^1$-homotopy theory was 
introduced in \cite{mv}. Its stabilization is considered in 
\cite{jardine.mss}. For technical reasons, the unstable projective
version  (which is the basis of \cite{dro.mf}) is more convenient,
although the closed motivic model structure 
described in \cite[Appendix]{ppr.bgl} seems to be quite ideal for the
comparison with the Grothendieck ring of varieties. 

A {\em base scheme\/} is a Noetherian separated
scheme of finite Krull dimension.
A {\em motivic space\/}
over $S$ is a presheaf on the site $\Sm_S$
of smooth separated $S$-schemes with values in the category of
simplicial sets. Let $\M(S)$ denote the category of pointed
motivic spaces.

\begin{example}
\label{ex:representable}
Any scheme $X$ in $\Sm_S$ defines a discrete
representable motivic space over $S$
which is also denoted $X$, and a discrete
representable pointed motivic space $X_+$ over $S$. One has
$X_+(Y)=\Set_{\Sm_S}(Y,X)_+$, where $B_+$ denotes
the set $B$ with a disjoint basepoint. Any (pointed) simplicial
set $L$ defines a constant (pointed) motivic space which is also
denoted $L$.
\end{example}

Many model structures exist on $\M(S)$ having the
Morel-Voevodsky $\Aff^1$-homotopy category of $S$ as
its homotopy category. Waldhausen's setup of algebraic
$K$-theory requires specific choices. 
The following model structure
is well-suited for base change (see \cite[Example 3.1.22]{mv}).

\begin{defn}
\label{defn:model-structure}
{\em Cofibrations\/} in $\M(S)$ are generated by the set
\begin{equation}\label{eq:I-S}
  \bigl\lbrace \bigl(X \times (\partial \Delta^n \hookrightarrow 
      \Delta^n)\bigr)_+\bigr\rbrace_{X\,\in \,\Sm_S, \,n\geq 0}.
\end{equation}
Applying the small object argument to this set produces
a cofibrant replacement functor $\kappa\colon (-)^c\rightarrow \Id_{\M(S)}$.
A pointed motivic space $B$ is {\em fibrant\/} if 
\begin{itemize}
\item $B(X)$ is a fibrant simplicial set for all $X\in \Sm_S$,
\item the image of every Nisnevich elementary distinguished square
  \[\xymatrix{
   V  \ar[r] \ar[d]    & Y  \ar[d]  \\
   U  \ar[r]    & X  }\]
  in $\Sm_S$ under $B$ is a homotopy pullback square of simplicial sets,
\item $B(\emptyset)$ is contractible, and 
\item for every $X\in \Sm_S$, the map $B(X\times \Aff^1 \xrightarrow{\mathrm{pr}} X)$
    is a weak equivalence of simplicial sets.
\end{itemize}
A map $\phi\colon D\rightarrow B$ of pointed motivic spaces over $S$ is a {\em weak
equivalence\/} if, for every fibrant motivic space $C$, the induced map
\[\sSet_{\M(S)}(\phi^c,C)\colon \sSet_{\M(S)}(B^c,C)\rightarrow \sSet_{\M(S)}(D^c,C)\]
is a weak equivalence of simplicial sets.
A map of motivic spaces is a {\em fibration\/} if it has the right lifting
property with respect to all cofibrations which are also weak equivalences
(the {\em acyclic\/} cofibrations). 
\end{defn}

\begin{theorem}\label{thm:model-structure-spaces}
The classes from Definition~\ref{defn:model-structure}
define a symmetric monoidal
$\sSet_\bullet$-model structure on $\M(S)$ which is
weakly finitely generated. It is
Quillen equivalent to the Morel-Voevodsky model.
\end{theorem}

\begin{proof}
See~\cite[Section 2.1]{dro.mf} for a proof. 
\end{proof}

\begin{remark}
\label{r:fibrations-fibrant}
The smash product of pointed motivic spaces is defined 
sectionwise.
The smash product of a weak equivalence with an
arbitrary pointed motivic space is a weak equivalence.
Since the domains and codomains of the generating cofibrations
are finitely presentable, a filtered colimit of weak equivalences
is again a weak equivalence. Moreover, a filtered colimit of
fibrant motivic spaces is again fibrant. 
\end{remark}

Proposition~\ref{prop:stabilize-model-structure}
applies to the model structure from Theorem~\ref{thm:model-structure-spaces}. 
The two relevant
examples are $S^1$-spectra $\Sp(S):=\Sp_{S^1}(\M(S))$
and $T$-spectra $\Sp_T(S):=\Sp_T(\M(S))$, as well 
as their symmetric analogues
$\SymSp(S)$ and $\SymSp_T(S)$.
Here $S^1=\Delta^1/\partial \Delta^1$
is the constant simplicial circle, and $T=S^1\smash S^{1,1,}$, 
where $S^{1,1}$ is the simplicial mapping cylinder of the unit
$S\hookrightarrow \mathbb{G}_m$ in the multiplicative group scheme over $S$. 

\begin{defn}\label{def:I-S}
  Let $S$ be a base scheme. Then $I_S$ denotes
  the set of generating cofibrations in $\M(S)$
  given in~(\ref{eq:I-S}),
  or (if no confusion can arise)
  the corresponding set of generating cofibrations
  in (symmetric) $B$-spectra over $S$ as introduced
  in the proof of Proposition~\ref{prop:stabilize-model-structure}
  and~\ref{prop:symm-spectra}, respectively.
\end{defn}

If $f\colon X\rightarrow Y$ is a morphism of base schemes, pullback along
$f$ defines a functor $\Sm_Y\rightarrow \Sm_X$. Precomposition with this
functor yields another functor, denoted $f_\ast \colon \M(X)\rightarrow \M(Y)$.
On objects 
\begin{equation}
\label{eq:f-lower-star}
(f_\ast B)(Z) = B(X\times_Y Z)
\end{equation}
for any $Z\in \Sm_Y$. 
Via left Kan extension, $f_\ast$ has a left adjoint
$f^\ast\colon \M(Y)\rightarrow \M(X)$
which is strict symmetric monoidal.
Since every motivic space is a colimit
of representable ones, $f^\ast$ is characterized by the formula
\begin{equation}
\label{eq:f-upper-star}
f^\ast(Z_+)=(X\times_Y Z)_+
\end{equation}
for every $Z\in \Sm_Y$. 

\begin{example}
\label{ex:internal-hom}
Base change describes the internal hom in $\M(X)$ as
\begin{equation*}
  \M(X)(C,B)(Z\xrightarrow{z} X) = \sSet_{\M(X)}(C,z_\ast z^\ast B).
\end{equation*}
Note that
if $f$ is smooth, the canonical natural transformation
\begin{equation}\label{eq:closed}
  f^\ast\M(Y)(C,B) \rightarrow  \M(X)(f^\ast C,f^\ast B) 
\end{equation}
is a natural isomorphism.
\end{example}

If $f\colon X\rightarrow Y$ is a smooth morphism of base schemes, composition
with $f$ defines a functor $\Sm_X \rightarrow \Sm_Y$. Precomposition with
this functor defines the functor $f^\star\colon \M(Y)\rightarrow \M(X)$, which then
has a left adjoint $f_\sharp \colon \M(X)\rightarrow \M(Y)$ 
by (enriched) Kan extension. Since every motivic space is a colimit
of representable ones, $f_\sharp$ is characterized by the formula
\begin{equation}
\label{eq:f-sharp}
f_\sharp(Z\xrightarrow{z} X)_+= (Z\xrightarrow{z} X\xrightarrow{f} Y)_+
\end{equation}
for every $Z\in \Sm_X$. If $Z\rightarrow Y$ is in $\Sm_Y$, the canonical 
$Y$-morphism $X\times_Y Z\rightarrow Z$ defines a map
$B(Z) \rightarrow f_\ast f^\star B (Z)$ which is natural in $Z$ and $B\in \M(Y)$,
and hence a natural transformation $\Id_{\M(Y)} \rightarrow f_\ast \circ f^\star$.

\begin{lemma}
\label{lem:smooth-base-change}
If $f\colon X\rightarrow Y$ is a smooth morphism of base schemes, the adjoint
\[ f^\ast \rightarrow f^\star \]
of the natural transformation $\Id_{\M(Y)} \rightarrow f_\ast \circ f^\star$
is a natural isomorphism.
\end{lemma}

\begin{proof}
This is straightforward.
\end{proof}

In the following, $f^\star$ will be used implicitly as a concrete
description for the left Kan extension $f^\ast$ whenever $f$ is smooth.
It has the advantage that it is strictly functorial. These base change
functors can be extended to the category of (symmetric)
$B$-spectra by levelwise application
in the case
$B\in \{S^1,T\}$. This extension involves 
the identification $f^\ast(B_Y)\xrightarrow{\iso} B_X$,
where $f\colon X\rightarrow Y$ and $B_S$ indicates that $B$ is a pointed
motivic space over $S$. They are still denoted by
$f_\ast\colon \Sp_B(X)\rightarrow \Sp_B(Y)$, etc.

\begin{proposition}\label{prop:strict-funct}
  Let $f\colon X\rightarrow Y$ and $g\colon Y\rightarrow Z$ be morphisms
  of base schemes.
  \begin{enumerate}
    \item\label{item:5} There is an equality $(g\circ f)_\ast = g_\ast \circ f_\ast$
      and a unique natural isomorphism $(g\circ f)^\ast \xrightarrow{\iso} f^\ast \circ g^\ast$.
    \item\label{item:6} If $f$ and $g$ are smooth, the unique natural isomorphism
      $(g\circ f)^\ast \xrightarrow{\iso} f^\ast \circ g^\ast$ is the identity, and there is a 
      unique natural isomorphism $(g\circ f)_\sharp \xrightarrow{\iso} g_\sharp \circ f_\sharp$.
    \item\label{item:7} There are equalities $\id_\ast = \Id$, $\id^\ast = \Id$, and
      a natural isomorphism $\id_\sharp\xrightarrow{\iso} \Id$.
    \item\label{item:8} The diagrams
      \[\xymatrix{
        \M(X) \ar[r]^-{f_\ast} \ar[d]_-{\Sigma^\infty_B} & \M(Y) 
        \ar[d]^-{\Sigma^\infty_B} & & \M(Y) \ar[r]^-{f^\ast} \ar[d]_-{\Sigma^\infty_B}& \M(X)  \ar[d]^-{\Sigma^\infty_B}\\
        \Sp(X,B) \ar[r]^-{f_\ast} & \Sp(Y,B) & \quad & \Sp(Y,B) \ar[r]^-{f^\ast} & \Sp(X,B)  \\
      }\]
      commute, and similarly for $f_\sharp$, and for symmetric spectra.
  \end{enumerate}
\end{proposition}

\begin{proof}
  This is straightforward. See also \cite[Chapitre 4]{ayoub.operations-2}.
\end{proof}

\begin{lemma}\label{lem:base-change-I}
  Let $f\colon X\rightarrow Y$ be a morphism of base schemes. 
  Then $f^\ast(I_Y)\subseteq I_X$, and
  $f_\sharp(I_X)\subseteq I_Y$ if $f$ is smooth.
\end{lemma}

\begin{proof}
  This follows from direct inspection.
\end{proof}

\begin{proposition}\label{prop:base-change-model}
  Let $f\colon X\rightarrow Y$ be a morphism of base schemes. Then $(f^\ast,f_\ast)$
  is a Quillen adjoint pair. If $f$ is smooth, then $(f_\sharp,f^\ast)$ is
  a Quillen adjoint pair.
\end{proposition}

\begin{proof}
  Consider the case of pointed motivic spaces first.
  Lemma~\ref{lem:base-change-I} implies that
  $f^\ast$ and (if $f$ is smooth) $f_\sharp$
  preserve cofibrations.
  To prove the first statement, it remains to show that $f_\ast$ preserves
  fibrations. By Dugger's lemma~\cite[A.2]{Dugger:replacing}, it suffices to prove
  that $f_\ast$ preserves fibrations between fibrant motivic spaces. These
  fibrations are detected by the set of acyclic cofibrations
  described in Remark~\ref{r:fibrations-fibrant}. Hence
  it suffices to prove that $f^\ast$ maps each of these special acyclic
  cofibrations in $\M(Y)$ to an acyclic cofibration in $\M(X)$. This is
  straightforward by equation~(\ref{eq:f-upper-star}).
  The proof for the second statement is similar, using
  equation~(\ref{eq:f-sharp}). 

  By Proposition~\ref{prop:base-change-model}, $f^\ast$
  preserves cofibrations. However, one can see directly
  that $f^\ast(\Fr I_Y) \subseteq \Fr I_X$ since $f^\ast$ commutes
  with the functors $\Fr_n$. As in the proof of
  the preceding case,
  it remains to prove that $f_\ast$ preserves stable fibrations
  of stably fibrant motivic $B_X$-spectra. Those coincide with the
  levelwise fibrations. Since $f_\ast $ preserves fibrations,
  it suffices to prove that $f_\ast$ preserves
  stably fibrant motivic $B_X$-spectra. This in turn follows 
  from the preceding case,
  since $f_\ast$ preserves weak equivalences of 
  fibrant pointed motivic spaces.

  If $f\colon X\rightarrow Y$ is smooth, $f_\sharp$ preserves cofibrations
  of motivic $B_X$-spectra. However, one can see directly
  that $f_\sharp(\Fr I_X) \subseteq \Fr I_Y$ since $f_\sharp$ commutes
  with the functors $\Fr_n$. 
  It remains
  to prove that $f^\ast$ preserves fibrations of stably fibrant motivic
  $B_Y$-spectra. As above, it suffices to check that $f^\ast$ preserves
  stably fibrant motivic $B_Y$-spectra. This follows from
  isomorphism~(\ref{eq:closed}),
  together with the fact that $f^\ast$ preserves all weak equivalences
  of pointed motivic spaces. The latter is implied by the fact that $f^\ast$
  is both a left and a right Quillen functor.
\end{proof}

\begin{lemma}\label{lem:base-change-finite}
  Let $f\colon X\rightarrow Y$ be a morphism of base schemes.
  The functor $f^\ast$ preserves finite objects,
  $I$-finite objects, and dualizable objects. If $f$ is smooth,
  $f_\sharp$ preserves finite objects and $I$-finite
  objects.
\end{lemma}

\begin{proof}
  The statements about $I$-finiteness appear already in
  the proof of Proposition~\ref{prop:base-change-model}.
  The implicit statements about cofibrancy follow from
  Proposition~\ref{prop:base-change-model}. 
  Observe that $f^\ast$ preserves finitely presentable
  objects since its right adjoint $f_\ast$ preserves filtered
  (even all!) colimits, and similarly for $f_\sharp$.
  Since $f^\ast$ is strict symmetric monoidal, it preserves
  dualizable objects.
\end{proof}

Let $\widetilde{\Sm}_S$ be the subcategory of $\Sm_S$ having the same objects,
but only smooth $S$-morphisms as morphisms.
One may summarize some of the results above by
saying that the model categories considered so far
are Quillen functors on $\widetilde{\Sm}_S$, but only
Quillen pseudo-functors on $\Sm_S$. It is possible to strictify
these Quillen pseudo-functors to a naturally (not just Quillen)
equivalent 
Quillen functor on
$\Sm_S$ by the categorical result \cite{power}. This will
be assumed from now on, without applying notational
changes.

\section{Algebraic $K$-theory of $\Aff^1$-homotopy theory}
\label{sec:algebraic-k-theory}

  Let $\M(S)$ be the model category of
  pointed motivic spaces over $S$, equipped
  with the $\mathbb{A}^1$-local Nisnevich
  model structure~\ref{defn:model-structure}.
  Before applying Waldhausen's $K$-theory construction,
  one of the
  finiteness notions 
  introduced in~\ref{defn:finiteness} will be imposed,
  indicated by the respective superscript
  $\M^\g(S)$ for $\g\in \{\f,\hf,\ifin,\ihf\}$.
  Unless otherwise specified, the $I$-finiteness
  notions always refer to the set
  of generating cofibrations listed
  in~\ref{defn:model-structure}.

\begin{defn}\label{defn:motivic-a-theory}
  Let $\g\in \{\f,\hf,\ifin,\ihf\}$. Then
  $\AAA(\M^\g(S))$ denotes the spectrum obtained by
  applying Waldhausen's $\Sdot$-construction to
  the Waldhausen category $\M^\g(S)$.
\end{defn}

Technically speaking, $\AAA(\M^\g(S))$ is the algebraic $K$-theory
of the one-point motivic space over $S$. It is
possible to consider the algebraic $K$-theory of an arbitrary
motivic space $B$ over $S$ by viewing the canonical
Waldhausen category of 
$\g$-finite motivic spaces over $S$ which are retractive over $B$,
as mentioned abstractly in Proposition~\ref{prop:retractive-model}.

\begin{proposition}[Waldhausen]\label{prop:apoint-retract}
  Let $\g\in \{\f,\hf,\ifin,\ihf\}$.
  The spectrum $\AAA(\M^\g(\mathbb{C}))$ contains $\AAA(\ast)$ as
  a retract. In particular, it is nontrivial.
\end{proposition}

\begin{proof}
  The constant pointed motivic space functor
  $\sSet_\bullet\rightarrow \M(S)$ and the complex realization functor
  $\M(\CC)\rightarrow \Top_\bullet$ are left Quillen functors.
  The constant pointed motivic space sends (homotopy) finite
  pointed simplicial sets to $I_\CC$-finite pointed motivic spaces.
  The complex realization functor sends representables to homotopy finite
  pointed topological spaces, and hence $I_\CC$-finite pointed motivic
  spaces to homotopy finite pointed topological spaces.
  A finite pointed motivic space is a retract of an $I_\CC$-finite
  pointed motivic space. Since homotopy finite pointed
  topological spaces are closed under retracts, the
  complex realization functor preserves homotopy finiteness.
  Hence both functors induce maps on
  Waldhausen $K$-theory spectra. Their
  composition coincides with the
  geometric realization functor
  \[ \lvert - \rvert\colon \sSet_\bullet \rightarrow \Top_\bullet, \]
  which induces an equivalence on Waldhausen
  $K$-theory by \cite[Theorem 2.1.5]{Waldhausen}.
  The statement follows.
\end{proof}

\begin{proposition}\label{prop:retract-general}
  Let $S$ be a base scheme and $\g\in \{\f,\hf,\ifin,\ihf\}$. 
  The spectrum $\AAA(\M^\g(S))$ contains $\AAA(\ast)$ as
  a retract. In particular, it is nontrivial.
\end{proposition}

\begin{proof}
  It suffices to consider a connected base scheme $S$.
  Let $\M_{\mathrm{hell}}(S)$ be the left Bousfield localization of the 
  Nisnevich local projective model structure with
  respect to the maps $X\rightarrow S$ in $\Sm_S$ such that $X$
  is connected. It is a left Bousfield localization of
  the $\Aff^1$-Nisnevich local projective model structure,
  and again weakly finitely generated.
  The identity functor is a left Quillen functor
  from $\M(S)$ to $\M_{\mathrm{hell}}(S)$ preserving
  finitely presentable cofibrant pointed motivic spaces.
  If $B \in\M_{\mathrm{hell}}(S)$ is fibrant, it is locally
  constant, since $B(S) \rightarrow B(X)$ is a weak equivalence
  for every smooth morphism $X\rightarrow S$ such that $X$ is connected.

  The constant pointed motivic space functor
  $\sSet_\bullet \rightarrow \M_{\mathrm{hell}}(S)$ is a left Quillen
  functor, but also a Quillen equivalence. Its right adjoint
  is the evaluation at the terminal scheme. A map of
  fibrant objects in $\M_{\mathrm{hell}}(S)$ is a weak equivalence
  if and only if it is a levelwise weak equivalence. Let
  $B\rightarrow C$ be a map of motivic spaces over $S$
  which are fibrant in $\M_{\mathrm{hell}}(S)$. If the
  map $B(S)\rightarrow C(S)$ is a weak equivalence, then $B(X)\rightarrow C(X)$
  is a weak equivalence for every connected $S$-scheme $X$.
  Since every smooth $S$-scheme admits a Zariski open cover
  by smooth connected $S$-schemes, $B(X)\rightarrow C(X)$
  is then a weak equivalence for every smooth $S$-scheme.
  It follows that 
  evaluation at the terminal scheme $S$ preserves and detects
  weak equivalences of fibrant objects in $\M_{\mathrm{hell}}(S)$.
  If $L$ is a pointed simplicial set, considered as a constant motivic
  space over $S$, its fibrant replacement in $\M_{\mathrm{hell}}(S)$
  sends $X$ to the product of $L$ indexed over the connected
  components of $X$. In particular, the derived unit of the 
  adjunction is the identity. This concludes the proof that
  the constant motivic space functor 
  $\sSet_\bullet \rightarrow \M_{\mathrm{hell}}(S)$ is a Quillen equivalence.

  Both the constant motivic space functor 
  $\sSet_\bullet \rightarrow \M(S)$ and the 
  identity functor $\M(S) \rightarrow \M_{\mathrm{hell}}(S)$ preserve
  finite and $I$-finite objects, hence induce maps
  on suitable Waldhausen categories. Since 
  $\sSet_\bullet \rightarrow \M_{\mathrm{hell}}(S)$ is a Quillen
  equivalence, it is a $K$-theory equivalence by \cite[Theorem 3.3]{Sagave}.
  The result follows.
\end{proof}

The full technology of Waldhausen's algebraic $K$-theory of spaces
requires that the weak equivalences satisfy the extension axiom.
This axiom is not satisfied in the category of pointed motivic
spaces (the counterexample given for pointed simplicial sets
in \cite[Section 1.2]{Waldhausen} right after the definition of
the extension axiom extends). However, it is satisfied in
the category of $S^1$-spectra of pointed motivic spaces over $S$.
The suspension spectrum functor induces a $K$-theory equivalence,
as one deduces from the following theorem. 

\begin{remark}\label{rem:stabilization-a}
  Theorem~\ref{thm:stabilization-A}  including its assertion
  for the $I$-finiteness notions show that  
  \[ \AAA^\g(S) := \AAA(\Sp^\g(S)) \leftarrow \AAA(\M^g(S)) \]
  is a $K$-theory equivalence for every $\g\in \{\f,\hf,\ifin,\ihf\}$.
  Moreover, it turns out
  to be natural in the base scheme $S$. Thus in the discussion below
  Waldhausen's fibration theorem may be applied to $\AAA^\g(S)$.
\end{remark}

As a consequence of Lemma~\ref{lem:base-change-finite}
and Proposition~\ref{prop:base-change-model}, 
the functors $f^\ast$ and (if applicable) $f_\sharp$ induce
exact functors on Waldhausen categories. 

\begin{proposition}\label{prop:open-split}
  Let $j\colon U\hookrightarrow S$ be an open embedding
  of base schemes, with reduced closed
  complement $i\colon Z\hookrightarrow S$. Then the functors 
  $j^\ast$ and $i^\ast$ induce a splitting
  \[ \AAA^\g(S)\xrightarrow{\sim} \AAA^\g(U)\times \AAA^\g(Z) \]
  for $\g\in\{\f,\hf,\ifin,\ihf\}$.
\end{proposition}

\begin{proof}
  Consider the homotopy fiber sequence
  \[ \hofib(j^\ast) \rightarrow \AAA^\g(S)\xrightarrow{j^\ast} \AAA^\g(U) \]
  of spectra. In order to identify the homotopy fiber of $j^\ast$, 
  the map $\AAA^\g(S)\xrightarrow{j^\ast} \AAA^\g(U) $
  is factored as follows. 
  Let $\mathrm{v}\Sp^\g(S)$ denote the subcategory of
  maps $f$ such that $j^\ast(f)$ is a weak equivalence in $\Sp(U)$.
  Let $\Sp^\g(S \vert U)$ denote the resulting Waldhausen category
  $(\Sp^\g(S),\ast,\mathrm{v}\Sp^\g(S),\mathrm{cof}\Sp^\g(S))$.
  The identity can then be regarded as an exact functor
  $\Phi\colon \Sp^\g(S)\rightarrow \Sp^\g(S \vert U)$.
  Almost by definition, $j^\ast \colon \Sp^\g(S \vert U) \rightarrow \Sp^\g(M(U))$
  satisfies the conditions of Waldhausen's Approximation Theorem
  \cite[Theorem 1.6.7]{Waldhausen}. In fact, $j^\ast$ detects and
  preserves weak equivalences by definition. 
  If $\E$ is a $\g$-finite $S^1$-spectrum
  over $S$ and $j^\ast(\E)\rightarrow \D$ is a map
  of $\g$-finite $S^1$-spectra over $U$, consider it as the
  map 
  \[ j^\ast(\E)=j^\ast j_\sharp j^\ast (\E)\rightarrow j^\ast j_\sharp (\D)=(\D). \]
  Here the fact that the unit
  $\Id\rightarrow j^\ast j_\sharp$ is the identity enters. The
  map $j_\sharp j^\ast (\E) \rightarrow j_\sharp (\D)$ can be factored
  via the simplicial mapping cylinder as a cofibration of
  $\g$-finite $S^1$-spectra over $U$, followed by a weak equivalence.
  Therefore, $j^\ast \colon \Sp^\g(S \vert U) \rightarrow \Sp^\g(M(U))$ satisfies
  the second approximation property, whence it is a $K$-theory
  equivalence by the Approximation Theorem \cite[Theorem 1.6.7]{Waldhausen}.
  Thus $\hofib(j^\ast)$ is weakly equivalent to the homotopy fiber
  of $\Phi$. The latter may be identified, by the Fibration Theorem
  \cite[Theorem 1.6.4]{Waldhausen}, with the algebraic $K$-theory
  of the sub-Waldhausen category $\Sp^{\g,j^\ast\simeq \ast}(S)$ 
  of $\g$-finite $S^1$-spectra $\E$ over
  $S$ such that $j^\ast(\E)$ is (weakly) contractible. The homotopy
  cofiber sequence
  \[ j_\sharp j^\ast (\E) \rightarrow \E \rightarrow i_\ast i^\ast(\E), \]
  which is an $S^1$-spectrum version of \cite[Theorem 3.2.21]{mv}, implies
  that the induced functor $i^\ast\colon \Sp^{\g,j^\ast\simeq \ast}(S)
  \rightarrow \Sp^\g(Z)$ satisfies the special approximation property.
  Thus $i^\ast$ induces a $K$-theory equivalence 
  by \cite[Theorem 2.8]{Sagave}.
  It remains to observe the splitting, which is induced by
  $j_\sharp\colon \M(U)\rightarrow \M(S)$, the left 
  adjoint of $j^\ast$. It is a left Quillen functor
  preserving the set of generating cofibrations given in~\ref{defn:model-structure}.
  The unit $\Id \rightarrow j^\ast j_\sharp$ is the identity, since
  $j$ is an open embedding. 
\end{proof}

In particular, the map $\AAA^\g(S)\rightarrow \AAA^\g(\Aff^1_S)$
induced by the projection is not a weak equivalence for
every $\g\in \{\f,\hf,\ifin,\ihf\}$, because
$\AAA^\g(\Aff^1_S\minus \{0\})$
is not contractible by Proposition~\ref{prop:retract-general}. 

\begin{corollary}\label{cor:A-t-loop}
  Let $\g\in\{\f,\hf,\ifin,\ihf\}$.
  There is a natural weak equivalence
  \[ \Omega_T\AAA^\g(S) \xrightarrow{\sim} \AAA^\g(S). \]
\end{corollary}

\begin{proof}
  This follows from the Yoneda lemma and
  Proposition~\ref{prop:open-split}.
\end{proof}

\begin{corollary}\label{cor:Nisnevich-descent}
  Let $\g\in \{\f,\hf,\ifin,\ihf\}$. Assume that the pullback square
  \begin{equation}
    \label{eq:distinguished-square}
    \xymatrix{
      V \ar[r] \ar[d] & X \ar[d]^-p\\
      U \ar[r]^-j & Y
    }
  \end{equation}
  in $\Sm_S$ 
  is either a Nisnevich distinguished square 
  or an abstract blow-up square.
  Then the square 
  \begin{equation}
    \label{eq:A-distinguished-square}
    \xymatrix{
      \AAA^\g(Y) \ar[r]^{j^\ast} \ar[d]_{p^\ast}& \AAA^\g(U) \ar[d]\\
      \AAA^\g(X) \ar[r] & \AAA^\g(V)
    }
  \end{equation}
  is a homotopy pullback square. 
\end{corollary}

\begin{proof}
  This is a straightforward consequence of Proposition~\ref{prop:open-split}.
\end{proof}

\begin{proposition}\label{prop:a-extends}
  Let $S$ be a base scheme and $\g\in \{\f,\hf,\ifin,\ihf\}$.
  Then Waldhausen $K$-theory provides a presheaf 
  \[ \AAA^\g\colon \Sm_S^\op\rightarrow \SymSp \]
  of symmetric $S^1$-spectra which is almost sectionwise fibrant. 
\end{proposition}

In fact,
the symmetric spectrum $\AAA^\g(X)$ is an 
$\Omega$-spectrum beyond the first term.
Corollary~\ref{cor:Nisnevich-descent} then implies that up to a sectionwise
fibrant replacement, the symmetric $S^1$-spectrum $\AAA^\g$ over $S$
is fibrant in the
Nisnevich local projective model structure. Its $\Aff^1$-fibrant
replacement can be determined fairly explicitly via Suslin's singular 
functor \cite{mv}. 
Recall the 
standard cosimplicial smooth
scheme $\Delta^\bullet_S$ over $S$ given by 
$[n]\mapsto \Delta^n_S = \Spec{\mathcal{O}_S[t_0,\dotsc,t_n]/\sum_{i=0}^n t_i =1}$.
Realizing the simplicial motivic $S^1$-spectrum $[n]\mapsto \AAA^\g(-\times_S \Delta^n)$
produces a motivic $S^1$-spectrum $\Sing \AAA^\g$ over $S$.
As Bj{\o}rn Ian Dundas pointed out to me, it is not very interesting.

\begin{proposition}\label{prop:a-fibrant}
  Let $\g\in \{\f,\hf,\ifin,\ihf\}$.
  A sectionwise fibrant replacement of the motivic $S^1$-spectrum 
  $\Sing\AAA^\g$ over $S$ is fibrant and sectionwise
  contractible.
\end{proposition}

\begin{proof}
  Recall first that the sectionwise fibrant replacement is fairly harmless, as
  the adjoint structure maps of $\AAA^\g(X)$ are weak equivalences, except
  for the first. The standard argument from \cite{mv} implies that
  $\Sing\AAA^\g(X\times \Aff^1_S\rightarrow X)$ is a stable equivalence. It remains
  to show that $\Sing\AAA^\g$ is still Nisnevich fibrant. Since $\emptyset
  \times_S\Delta^n_S$ is the empty scheme, $\Sing\AAA^\g(\emptyset)$
  is the realization of a degreewise contractible spectrum, hence contractible.
  The value of $\Sing\AAA^\g$ at a 
  distinguished square $Q$ as displayed in~(\ref{eq:distinguished-square})
  is a homotopy pullback
  square by \cite[Appendix B]{bousfield-friedlander}. In fact, 
  Corollary~\ref{cor:Nisnevich-descent} implies that in every simplicial degree
  $\AAA^\g(Q\times_S \Delta^n_S)$ is a homotopy pullback square. To apply
  \cite[Appendix B]{bousfield-friedlander}, it remains
  to observe that, by construction, every pointed simplicial set
  occurring in $\AAA^\g(Q\times_S \Delta^n_S)$ is connected. 

  For the final statement, observe that for every closed embedding
  $i\colon Z\hookrightarrow X$ in $\Sm_X$, with open complement
  $j\colon U=X\minus Z \hookrightarrow X$, the induced map
  \[ (i^\ast,j^\ast)\colon \Sing\AAA^g(X)\to \Sing \AAA^\g(Z)\times \Sing \AAA^\g(U)\]
  is a weak equivalence, by realizing the weak equivalences from
  Proposition~\ref{prop:open-split}. Suppose that $B\colon \Sm_S\to \Sp$ is
  any presheaf of spectra with this property which is also $\Aff^1$-invariant.
  Then $B(X) \xrightarrow{\sim} B(\Aff^1_X)\sim B(\Aff^1_X\minus \{0\})\times B(X)$ is an equivalence, as well as the identity on the second factor.
  Hence its cofiber $B(\Aff^1_X\minus\{0\})$ is contractible. It contains
  $B(X)$ as a retract included via $1\colon X\to \Aff^1_X\minus \{0\}$.
  Thus $B(X)$ is also contractible.
\end{proof}

The contractibility of the $\Aff^1$-homotopy type associated with
the Waldhausen $K$-theory of any of the finiteness notions on motivic
homotopy theory stated
in Proposition~\ref{prop:a-fibrant} calls for an adjustment. A
solution is the finiteness notion of dualizability, as introduced
in Definition~\ref{def:dualizable}. Since dualizability results
for smooth varieties in motivic homotopy theory involve invertibility
of Thom spaces of vector bundles, passage to $T$-spectra is required.

\begin{defn}\label{def:a-t-spectra}
  Let $\g\in \{\du,\f,\hf,\ifin,\ihf\}$
  be one of the finiteness notions introduced
  above, and let $S$ be a base scheme. Set
  \[ \AAA_T^\g(S):=\AAA(\SymSp^\g_T\M(S)) \]
  the algebraic $K$-theory of the category of $\g$-finite
  symmetric $T$-spectra over $S$.
\end{defn}

It is straightforward to verify that 
$\AAA_T^\g(S)$ satisfies similar properties as $\AAA^\g(S)$.
More precisely, Proposition~\ref{prop:apoint-retract} holds as well for
$\AAA_T^\g(S)$ and any $\g\in \{\du,\f,\hf,\ifin,\ihf\}$. Furthermore,
for any $\g\in \{\f,\hf,\ifin,\ihf\}$,
Propositions~\ref{prop:open-split},~\ref{prop:a-extends}, and~\ref{prop:a-fibrant},
and Corollaries~\ref{cor:Nisnevich-descent} and~\ref{cor:A-t-loop} hold with $\AAA_T^\g$ replacing $\AAA^\g$.
However, $\AAA_T^\du$ does not lead to a contractible $\Aff^1$-homotopy type,
as Theorem~\ref{thm:s-trace} implies. Nevertheless, the global sections
of $\AAA_T^\du$ and $\AAA_T^\hf$ coincide over fields of characteristic zero.

\begin{proposition}\label{prop:finite-dualizable}
  Let $F$ be a field of characteristic zero. Then
  the Waldhausen categories
  $\SymSp_T(F)^\hf$ and $\SymSp_T(F)^\du$ coincide.
\end{proposition}

\begin{proof}
  This follows from the main result in \cite{riou}. Here are some details.
  Smooth projective schemes are dualizable over a base scheme
  \cite[Theorem A.1]{Hu:Picard}, \cite[Chapitre 4]{ayoub.operations-2}.
  If $F$ is a field of characteristic zero, resolution
  of singularities provides that then also smooth quasiprojective
  $F$-varieties are dualizable \cite{Voevodsky:open}, 
  \cite[Theorem 52]{RO:mz}. Every smooth $F$-variety admits a Zariski
  open cover by smooth quasiprojective $F$-varieties, which
  implies that smooth $F$-varieties are dualizable.
  Hence every $I_F$-finite symmetric $T$-spectrum is dualizable.
  Since the property of being dualizable is closed under retracts
  and weak equivalences, and every finite cofibrant symmetric
  $T$-spectrum is a retract of
  an $I_F$-finite symmetric $T$-spectrum by Lemma~\ref{lem:fin-ifin},
  every homotopy finite symmetric $T$-spectrum is dualizable. 
  
  Conversely, every dualizable symmetric $T$-spectrum is compact
  as an object of $\SH(F)$. It is an easy consequence of basic
  properties of dualizable objects in a symmetric monoidal 
  stable model category with a compact unit~\cite{May:traces}. 
  Any compact $T$-spectrum
  in $\SH(F)$ is a retract of a cofibrant $T$-spectrum which
  is weakly equivalent to an $I$-finite $T$-spectrum. 
\end{proof}

Proposition~\ref{prop:finite-dualizable} holds also over fields of
positive characteristic, provided that the characteristic is 
inverted. However, it does not hold for base schemes of positive
dimension. For example, the $T$-suspension spectrum of 
$\Aff^1_\CC\minus\{0\}$ is not dualizable in $\SH(\Aff^1_\CC)$, although
it is finite; see \cite[Remark 8.2]{nso.landweber} for a more general statement.

\section{Grothendieck rings}
\label{sec:grothendieck-rings}

Let $S$ be a base scheme. The {\em Grothendieck ring of\/} $S$ is the
free abelian group on isomorphism classes of finite-type $S$-schemes,
denoted $[X]$,
modulo the relations $[X] = [Z]+ [X\minus Z]$ whenever $Z$ is a closed
subscheme, and $[\emptyset]=0$. The ring structure is induced by
the product $X\times_S Y$. The ring
$K_0(\Var_S)$ is commutative and has $[S]$ as a unit.
Note that $[X]=[X_\mathrm{red}]$, where $X_\mathrm{red} \hookrightarrow X$
denotes the maximal reduced closed subscheme. Weak factorization
is used in the following main result of \cite{bittner}, which
gives a much simpler presentation of the Grothendieck ring
of fields of characteristic zero.

\begin{theorem}[F.~Heinloth n\'ee Bittner]\label{thm:bittner}
  Suppose that $F$ is a field of
  characteristic zero. Then $K_0(\Var_F)$ is generated by isomorphism
  classes of connected smooth projective $F$-schemes, modulo the relations
  $[X]-[f^{-1}(Z)] = [Y]-[Z]$ whenever $f\colon X\rightarrow Y$ is the 
  blow-up of the smooth projective variety $Y$ along the smooth
  center $Z\hookrightarrow Y$, and $[\emptyset]=0$.
\end{theorem}

A {\em motivic measure\/} on $S$ is a ring homomorphism
\[ K_0(\Var_S)\rightarrow A\]
to some commutative ring. Main examples of motivic measures are the Euler
characteristic $K_0(\Var_\CC)\rightarrow \ZZ$ on the
complex numbers, point counting
$K_0(\Var_{\FF_q})\rightarrow \ZZ$ on a finite field, and the Gillet-Soul\'e 
motivic measure \cite{gillet-soule}.
Theorem~\ref{thm:bittner} simplifies
the construction of motivic measures. For example, the motivic measure
on fields of characteristic zero obtained by sending a smooth projective
variety to its stable birational class
constructed in \cite{larsen-lunts.motivic} can be deduced from
Theorem~\ref{thm:bittner}.
In order to provide a new motivic measure, recall that if $\mathbf{C}$ is
a Waldhausen category, the abelian group $\pi_0\AAA(\mathbf{C})$ is
generated by the objects in $\mathbf{C}$, subject to the
following two relations:
\begin{enumerate}
\item $\langle B\rangle = \langle C\rangle$ if there exists a
  weak equivalence $B\xrightarrow{\sim} C$
\item $\langle B\rangle +\langle D\rangle = \langle C\rangle$ if there exists a
  cofibration $B\hookrightarrow C$ with cofiber $D$
\end{enumerate}

\begin{theorem}\label{thm:a-motivic-measure}
  Let $F$ be a field of characteristic zero.
  Sending the isomorphism class $[X]$ 
  of a smooth projective $F$-scheme $X$ to 
  its class $\langle X_+\rangle\in \pi_0 \AAA^\ifin(F)$
  defines a surjective motivic measure
  \[ \Phi_F\colon K_0(\Var_F) \rightarrow \AAA^\ifin(F). \]
\end{theorem}

\begin{proof}
  The relations given in Theorem~\ref{thm:bittner} 
  are fulfilled in $\pi_0 \AAA^\ifin(F)$ by
  \cite[Remark 3.2.30]{mv}. Hence $[X]\mapsto \langle X_+\rangle$
  defines a group homomorphism
  \[ \Phi_F\colon K_0(\Var_F) \rightarrow \pi_0\AAA^\ifin(F), \]
  which is compatible with the multiplicative structure.
  It remains to prove its surjectivity. However, $\pi_0\AAA^\ifin(F)$
  is generated as an abelian group by $I$-finite $S^1$-spectra 
  over $F$, and hence by the domains and codomains of the
  maps in $I$. These are of the form 
  $\Fr_mX_+\smash \partial \Delta^n_+$ and $\Fr_mX_+\smash \Delta^n_+$, 
  where
  $X$ is a smooth $F$-variety and $m,n\in \NN$. 
  Since $\Fr_m$ corresponds to a simplicial desuspension and
  suspension induces multiplication by $-1$ on $\pi_0\AAA^\ifin(F)$,
  one may restrict to $m=0$. Induction on $n$ and the cofiber sequence
  \[ X_+\smash \partial \Delta^n_+\hookrightarrow X_+\smash \Delta^n_+\rightarrow X_+\smash S^n \]
  imply that $\pi_0\AAA^\ifin(F)$ is generated as an abelian group by 
  $S^1$-suspension spectra $\Fr_0X_+ = \Sigma^\infty X_+$
  of smooth $F$-varieties.
  Resolution of singularities implies that it suffices to 
  restrict to $S^1$-suspension spectra of
  smooth projective $F$-varieties, similar to the argument
  in the proof of Proposition~\ref{prop:finite-dualizable}. 
  This concludes the proof.
\end{proof}

The formula 
$\Phi_F([X]) =\langle X_+\rangle$ does not apply to non-projective
varieties in general. For example,
\[ \Phi_F([\Aff^1]) = \Phi_F([\PP^1]-[\Spec{F}])=\langle \PP^1_+\rangle-\langle\Spec{F}_+\rangle = 
   \langle\PP^1,\infty \rangle \neq \langle\Aff^1_+\rangle = \langle\Spec{F}_+\rangle, \]
where the inequality follows -- at least in the case $k\subseteq \mathbb{R}$ --
from the left Quillen functor which takes a smooth $k$-scheme to the
topological space of its complex points, together with the
conjugation action. The fixed points of the action on the left hand side
of the inequality yield the class of $\mathbb{RP}^1$ having reduced Euler
characteristic $-1$, while the fixed points of the action on the right hand side
of the inequality have reduced Euler characteristic $1$.

Since $I$-finiteness is the smallest of the finiteness notions
$\g\in \{\f,\hf,\ifin,\ihf\}$
considered on motivic spaces over a field,
there is a motivic measure 
\[ \Phi_F\colon K_0(\Var_F) \rightarrow \pi_0\AAA^\g(F) \]
as well; however, it may not be surjective in the case $\g\in \{\f,\hf\}$.
Proposition~\ref{prop:apoint-retract} shows that it refines
the Euler characteristic if $F$ is a subfield of $\CC$. It
also refines the Gillet-Soul\'e ``motivic'' motivic measure \cite{gillet-soule}.

\begin{proposition}\label{prop:gillet-soule}
  Let $F$ be a field of characteristic zero.
  There is a commutative diagram
  \[\xymatrix{
    K_0(\Var_F) \ar[r]^{\Phi_F} \ar[d]_{\Psi_F} & \pi_0 \AAA^\ifin(F) \ar[d]\\
    K_0(\mathrm{ChMot}^{\mathrm{eff}}_F) \ar[r]^\iso & 
    K_0(\mathrm{DM}^{\mathrm{eff},\hf}_F)}\]
  of ring homomorphisms, where $\Psi_F$ maps
  the class of a smooth projective $F$-scheme to the class of
  its effective Chow motive.
\end{proposition}

\begin{proof}
  Voevodsky's derived category of effective motives
  may be obtained as the homotopy category of 
  $S^1$-spectra of motivic spaces with transfers; see \cite{RO:mz}. 
  This implies an identification of
  $K_0(\mathrm{DM}_F^{\mathrm{eff},\hf})$ with the ring of path components
  of $\AAA(\M^{\mathrm{tr},\hf}(F))$, where 
  $\M^{\mathrm{tr}}(F)$ is the model category
  of motivic spaces with transfers defined via the functor
  $\M^{\mathrm{tr}}(F)\rightarrow \M(F)$ forgetting transfers.
  Its left adjoint induces the vertical arrow on the
  right hand side of the diagram displayed above.
  Similar to the argument in the proof of 
  Proposition~\ref{prop:finite-dualizable} is an
  argument proving that homotopy finite and compact
  motives coincide in $\mathrm{DM}_F^{\mathrm{eff}}$.
  Furthermore, compact objects and
  geometric motives coincide by \cite[Theorem 11.1.13]{cisinski-deglise.dm},
  both in the effective and the non-effective case.
  The lower horizontal morphism in Proposition~\ref{prop:gillet-soule}
  is thus an isomorphism by \cite[Theorem 6.4.2]{bondarko.voevodsky-hanamura}.
\end{proof}

\begin{theorem}\label{thm:a-motivic-measure-2}
  Let $F$ be a field of characteristic zero.
  The ring homomorphism $\Phi_F$ extends to a 
  surjective ring homomorphism
  \[ \Phi_F\colon K_0(\Var_F)[\mathbb{L}^{-1}] \rightarrow \pi_0 \AAA_T^\ifin(F), \]
  where $\mathbb{L} = [\Aff^1_S]$.
\end{theorem}

\begin{proof}
  This follows from Theorem~\ref{thm:a-motivic-measure},
  the equality $\Phi_F(\mathbb{L}) = \langle \PP^1,\infty \rangle$,
  and the fact that $\Sigma_T^\infty(\PP^1,\infty)$ is invertible
  in $\SH(F)$, hence also in $\pi_0 \AAA_T^\ifin(F)$.
\end{proof}

In particular, all relations that hold in the Grothendieck ring
$K_0(\Var_F)$ or its localization $K_0(\Var_F)[\mathbb{L}^{-1}]$ also hold
in $\pi_0 \AAA(F)$ or $\pi_0\AAA_T(F)$, respectively. The ring
$K_0(\Var_F)[\mathbb{L}^{-1}]$ is relevant to the theory of motivic integration,
and also to the construction of the duality involution induced
by $[X]\mapsto \LL^{-\dim{X}}[X]$. In the ring $\pi_0\AAA_T^\g(F)$
(which from a certain perspective still consists of algebro-geometric 
objects),
the class of the pointed projective line is naturally invertible. 
Also the duality involution has a natural interpretation
in $\pi_0\AAA_T^\g(F)$, since the dual of a smooth projective $F$-variety $X$
is the Thom $T$-spectrum of its negative tangent bundle, considered as a class in $K_0(X)$. The
equality
\[ \langle \Th{-\Tan{X}}\rangle = \langle \PP^1,\infty\rangle^{-\dim{X}}
\cdot \langle X_+\rangle \]
follows from the Zariski local triviality of vector bundles.
However, the localization passage $K_0(\Var_F)\rightarrow K_0(\Var_F)[\LL^{-1}]$ involves
a loss of information.
Lev Borisov proved that $\LL$ is a zero divisor in $K_0(\Var_F)$ \cite{borisov}.
In particular, the composition
\[ K_0(\Var_F) \rightarrow \pi_0 \AAA^\ifin(F) \rightarrow \pi_0\AAA_T^\ifin(F) \]
is not injective.

\begin{proposition}\label{prop:inverting-t}
  The canonical homomorphism
  $\pi_0 \AAA^\g(F) \rightarrow \pi_0\AAA_T^\g(F)$ induces a 
  surjective homomorphism
  $\pi_0\AAA^\g(F)[\langle \PP^1,\infty \rangle^{-1}]\rightarrow \pi_0\AAA_T^\g(F)$
  for $\g\in \{\ifin,\f,\ihf,\hf\}$. 
\end{proposition}

\begin{proof}
  In fact, this holds for any
  base scheme $S$. 
  The abelian group $\pi_0\AAA_T^\g(F)$
  is generated by shifted $T$-suspension spectra $\Fr_m X_+$ 
  of smooth $S$-schemes; the
  contribution from the simplicial direction can be ignored, as the
  proof of Theorem~\ref{thm:a-motivic-measure} implies. The symmetric
  $T$-spectrum $\Fr_m X^+ \smash T^m$ is stably equivalent to
  $\Fr_0 X_+ = \Sigma^\infty_T X_+$, showing that 
  \[\pi_0 \AAA^\g(S)[\langle \PP^1,\infty \rangle^{-1}] \rightarrow \pi_0\AAA_T^\g(S) \]
  is surjective. 
\end{proof}

More can be and has been said on the relationship between
the Grothendieck ring of varieties and of motives.
Proposition~\ref{prop:gillet-soule} and Theorem~\ref{thm:a-motivic-measure-2}
provide the commutative diagram
\[\xymatrix{
  K_0(\Var_F) \ar[r]^{\Phi_F} \ar[d]_{\Psi_F}& \pi_0 \AAA^\ifin(F) \ar[r] \ar[d]& \pi_0 \AAA_T^\ifin(F)\ar[d]\\
  K_0(\mathrm{ChMot}^{\mathrm{eff}}_F) \ar[r]^\iso & 
  K_0(\mathrm{DM}^{\mathrm{eff},\hf}_F)\ar[r] & K_0(\mathrm{DM}^{\hf}_F)}\]
in which the homomorphism 
$K_0(\mathrm{DM}^{\mathrm{eff},\hf}_F)\rightarrow K_0(\mathrm{DM}^{\hf}_F)$
corresponds to inverting $\LL$, the class of the Lefschetz motive.
The latter can be deduced from
Voevodsky's cancellation theorem \cite{voevodsky.cancellation}.
If one imposes rational instead of integral coefficients on the
$T$-spectra and motives above, the canonical functor becomes
a Quillen equivalence for all fields in which $-1$ is a sum of squares,
by a theorem of Morel's. It is known that the canonical homomorphism
\[ K_0(\Var_F)[\LL^{-1}] \rightarrow K_0(\mathrm{DM}^{\hf}_F\otimes \QQ) \]
is not injective \cite[Proposition 7.9]{nicaise.trace}.
Already inverting $2$ in the homotopy category of $T$-spectra
splits it as a product $\SH(F)_+\times \SH(F)_-$ corresponding
to the two idempotents $\frac{1-\varepsilon}{2},\frac{1+\varepsilon}{2}$.
Here $\varepsilon$ is induced by the twist isomorphism on 
$T\smash T$. If $F$ is formally real, the
category $\SH(F)_-$ maps nontrivially to the derived category
of $\ZZ[\tfrac{1}{2}]$-modules. After rationalizing, the category $\SH(F)_+$
is equivalent to the derived category of motives over any field 
$F$ \cite[Theorem 16.2.13]{cisinski-deglise.dm}.
In particular, the canonical homomorphism
\[ \pi_0 \AAA(\mathbf{Sp}^\hf_T(F)\tensor \QQ) \rightarrow K_0(\mathrm{DM}^{\hf}_F \tensor \QQ) \]
is always surjective, but not injective if $F$ is formally real.
See \cite[Theorem 1.5]{jin.real} for an identification of
$\pi_0 \AAA(\mathbf{Sp}^\hf_T(X)\tensor \QQ)$ in terms
of Chow motives and the real \'etale site of the excellent and
separated scheme $X$ of finite Krull dimension.

\section{A trace map}
\label{sec:trace}

The next goal is to produce a trace map on  
the $\Aff^1$-homotopy type $\AAA_T^\du\colon \Sm_F^\op \rightarrow \Sp$ for a field $F$
of characteristic zero and the duality finiteness notion. The
general result \cite[Theorem 6.5]{hss.higher-traces} essentially
provides such a trace. However, when the existence of the motivic trace was
announced at a talk in Heidelberg in 2014, 
the argument proceeded along the lines 
of \cite{vogell}. For the sake of concreteness, this construction of the
trace will be sketched as follows.
In principle, it suffices to fatten the Waldhausen category $\SymSp_{T}^\du(F)=\SymSp_{T}^\hf(F)$ slightly as in \cite{vogell}. 
The fattened Waldhausen category consists of additional duality
data.

\begin{defn}\label{def:spectra-duality-data}
  For a base scheme $S$, let $\mathbf{DSp}(S)$ be the category whose objects
  are triples $(\E^+,\E^-,e^-)$, where
  $\E^+$ is a dualizable symmetric $T$-spectrum,
  $\E^-$ is fibrant symmetric $T$-spectrum, and
  $e^-\colon \E^-\smash \E^+\rightarrow \fib(\one)$
  is a map whose adjoint $\E^-\rightarrow \inthom(\E^+,\fib(\one))$
  is a weak equivalence. A morphism of triples from
  $(\D^+,\D^-,d^-)$ to $(\E^+,\E^-,e^-)$ is a pair
  $\phi^+\colon \D^+\rightarrow \E^+$, $\phi^-\colon \E^-\rightarrow \D^-$
  of maps such that the diagram
  \[\xymatrix{
    \D^-\smash \E^+ \ar[r]^{\phi^-\smash \E^+} \ar[d]_{\D^-\smash \phi^+}& 
    \E^-\smash \E^+ \ar[d]^ {e^-} \\
    \D^-\smash \D^+ \ar[r]^{d^-} & \fib(\one)
  }\]
  commutes. Such a morphism $(\phi^+,\phi^-)$ is 
  a {\em weak equivalence\/} if
  both $\phi^+$ and $\phi^-$ are weak equivalences,
  and a {\em cofibration\/} if $\phi^+$ is a cofibration
  and $\phi^-$ is a fibration.
\end{defn}

Since smashing with a cofibrant symmetric $T$-spectrum preserves
weak equivalences \cite[Proposition 4.19]{jardine.mss}, the 
symmetric $T$-spectrum $\E^-\smash \E^+$ has the correct homotopy type,
even if $\E^-$ is not cofibrant.

\begin{proposition}\label{prop:duality-waldhausen}
  The category $\mathbf{DSp}(S)$ is a Waldhausen
  category. 
\end{proposition}

\begin{proof}
  The category $\mathbf{DSp}(S)$ is pointed by 
  $(\ast,\ast,\ast \to \fib(\one))$.
  Weak equivalences in $\mathbf{DSp}(S)$ form a subcategory,
  and so do the cofibrations. Every triple is then
  cofibrant, using that the second entry is 
  fibrant.
  Suppose that 
  \[(\B^+,\B^-,b^-) \xleftarrow{(\psi^+,\psi^-)} (\D^+,\D^-,d^-)\xrightarrow{(\phi^+,\phi^-)} (\E^+,\E^-,e^-) \]
  is a diagram. Its pushout is defined as the triple
  $(\E^+\cup_{\D^+}\B^+,\E^-\times_{\D^-}\B^-,c)$ where $c$ is adjoint
  to the map
  \begin{equation}\label{eq:duality-waldhausen} 
    \E^-\times_{\D^-}\B^- \rightarrow \inthom(\E^+\cup_{\D^+}\B^+,\fib\one)
    \iso \inthom(\E^+,\fib\one)\times_{\inthom(\D^+,\fib\one)}
    \inthom(\B^+,\fib\one) 
  \end{equation}
  induced by the adjoints of $b^-$, $d^-$, and $e^-$. 
  The dual of the gluing lemma implies that 
  the map~(\ref{eq:duality-waldhausen}) is a weak equivalence.
\end{proof}

\begin{lemma}\label{lem:duality-waldhausen}
  The forgetful functor $\mathbf{DSp}(S) \rightarrow \SymSp_T^\du(S)$
  is a $K$-theory equivalence.
\end{lemma}

\begin{proof}
  The forgetful functor sends the triple $(\E^+,\E^-,e^-)$
  to $\E^+$ and is exact by definition. It admits the exact
  section
  sending $\E$ to the triple $(\E,\inthom(\E,\fib\one),\mathrm{ev})$
  where $\mathrm{ev}\colon \inthom(\E,\fib\one)\smash \E\rightarrow \fib\one$
  is the evaluation map, adjoint to the identity. Moreover,
  there is a natural weak equivalence
  \[  (\E^+,\inthom(\E^+,\fib\one),\mathrm{ev}) \xrightarrow{(\id,\flat(e^-))} 
  (\E^+,\E^-,e^-) \]
  where $\flat(e^-)$ is the adjoint of $e^-$. Hence applying
  the section to the forgetful functor induces a map on
  algebraic $K$-theory which is homotopic to the identity map.
\end{proof}

Recall that $\SymSp_T(S)$ admits the structure of a pointed
simplicial model category in which the $n$-simplices of morphisms 
are given by the maps $\D\smash \Delta^n_+ \rightarrow \E$ of symmetric
$T$-spectra over $S$. The functoriality listed in 
Proposition~\ref{prop:base-change-model} is simplicial. It follows
that the assignment $[n]\mapsto \SymSp_T^\du(S)_n$ is a simplicial
category with constant objects, and so is the assignment 
$[n]\mapsto \mathrm{w}\SymSp_T^\du(S)_n$
restricted to the subcategories of weak equivalences.

\begin{proposition}\label{prop:duality-waldhausen-simplex}
  The category $\mathbf{DSp}(S)_\bullet$ is a simplicial 
  category, and so is the restriction to 
  $\mathrm{w}\mathbf{DSp}(S)_\bullet$, the subcategory of
  weak equivalences.
\end{proposition}

\begin{proof}
  The $n$-simplices of morphisms are given by
  pairs $(\D^+\smash \Delta^n_+ \rightarrow \E^+,\E^-\smash \Delta^n_+\rightarrow \D^-)$
  satisfying the appropriate compatibility condition. 
  The required axioms are straightforward to check.
\end{proof}

\begin{lemma}\label{lem:duality-waldhausen-simplicial}
  The natural inclusion
  \[ \mathrm{w}\mathbf{DSp}(S) \xrightarrow{\kappa}
  \mathrm{w}\mathbf{DSp}(S)_\bullet 
   \]
  induces a weak equivalence after geometric realization.
\end{lemma}

\begin{proof}
  Lemma~\ref{lem:duality-waldhausen}, or rather
  its proof, implies that the forgetful functor induces the following
  diagram
  \[\xymatrix{
    \mathrm{w}\mathbf{DSp}(S) 
    \ar[r]^\kappa \ar[d]&
    \mathrm{w}\mathbf{DSp}_\bullet(S) \ar[d] \\
     \mathrm{w}\mathbf{SymSp}^\du_T(S) \ar[r]^\kappa &
    \mathrm{w}\mathbf{SymSp}^\du_T(S)_\bullet 
  }\]
  whose vertical arrows induce weak equivalences after geometric
  realization. It suffices to prove that the lower horizontal
  arrow has the same property.
  The inclusion $\kappa$ is induced by the collection of
  degeneracy maps $s_m\colon \Delta^m \rightarrow \Delta^0$. 
  Let $d_m\colon \Delta^0\rightarrow \Delta^m$ be the inclusion
  of the $m$-th vertex. By the realization lemma, it suffices to prove that the
  composition 
  \[  \mathrm{w}\mathbf{SymSp}^\du_T(S)_m \xrightarrow{d_m^\ast}
  \mathrm{w}\mathbf{SymSp}^\du_T(S) \xrightarrow{s_m^\ast} \mathrm{w}\mathbf{SymSp}^\du_T(S)_m\]
  is homotopic to the identity for every $m$. This follows from the fact
  that $\Delta^m$ simplicially contracts onto its last vertex.
\end{proof}

Let $(\E^+,\E^-,e^-)$ be an object in $\mathbf{DSp}(S)$.
Consider the simplicial set of maps of symmetric $T$-spectra
over $S$ from $\one$ to $\fib(\E^+\smash \E^-)$. The aim is to
modify this simplicial set to consist of only those maps
which -- together with $e^-$ -- 
express $\E^+$ and $\E^-$ as dual objects in the 
stable homotopy category $\SH(S)$. 
A map $\one \rightarrow \fib(\E^+\smash \E^-)$ induces a composition
\[ \E^+=\one \smash \E^+ \rightarrow \fib(\E^+\smash \E^-)\smash \E^+
\xrightarrow{\sim} \fib(\E^+\smash \E^-\smash \E^+) \xrightarrow{\fib(\E^+\smash e^-)}
\fib(\E^+\smash \fib(\one)) \] 
(as well as a similar composition 
$\cof(\E^-)\rightarrow \fib(\fib(\one)\smash \E^-)$ where $\cof$
is a cofibrant replacement functor).
There is a preferred map $z\colon\E^+ \rightarrow \fib(\E^+\smash \fib(\one))$, 
given by unit and
replacement natural transformations. Let
\[ H(\E^+,\E^-,e^-) =\bigl\lbrace 
H\colon \E^+\smash \Delta^1_+\to \fib(\E^+\smash \fib(\one))
\mathrm{\ with\ } H{\vert}_{\E^+ \smash 0_+}= z \bigr\rbrace \]
be the simplicial set of 
simplicial homotopies starting at the respective preferred map. 
By construction, $H(\E^+,\E^-,e^-)$ is a fibrant
simplicial set which simplicially contracts to the zero simplex
given by the preferred map. Moreover, it maps via 
an ``endpoint'' Kan fibration to the 
simplicial set of maps in the terminal corner of
the following diagram whose pullback is the
desired modification:
\begin{equation}\label{eq:modification}
  \xymatrix{
  D(\E^+,\E^-,e^-) \ar[r] \ar[d] &  \sSet(\one,\fib(\E^+\smash \E^-)) \ar[d]\\
  H(\E^+,\E^-,e^-)  \ar[r] &  \sSet(\E^+,\fib(\E^+\smash \fib\one))}
\end{equation}
The condition on $e^-$ guarantees that the vertical map on the right
hand side of diagram~(\ref{eq:modification}) is a weak equivalence,
and the horizontal arrows depict fibrations.
Hence $D(\E^+,\E^-,e^-)$ is a contractible fibrant
simplicial set. Its zero simplices are 
maps $e^+\colon \one\rightarrow \fib(\E^+\smash \E^-)$, together with a simplicial homotopy providing that $\E^+\smash e^- \circ e^+\smash \E^+$
coincides with $\id_{\E^+}$ in the motivic stable homotopy category of $S$.
Such a zero simplex
represented
by the tuple 
\[ ((\E^+,\E^-,e^-),e^+,H) \]
maps naturally to the composition 
\[ \one \xrightarrow{e^+} \fib(\E^+\smash \E^-)
\xrightarrow{\iso} \fib(\E^-\smash\E^+)
\xrightarrow{\fib(e^-)} \fib(\fib(\one)), \]
which is a zero simplex in $\fib(\fib(\one))(S)$
representing the Euler characteristic of $\E^+$ \cite{May:traces}.
More generally, an $n$-simplex maps to an $n$-simplex
in $\fib(\fib(\one))(S)$. 
A similar variant $D(\underline{\E})$ exists for an $n$-simplex
\[\underline{\E} =  (\E_0\xrightarrow{\sim} \E_1\xrightarrow{\sim} \dotsm \xrightarrow{\sim} \E_n) \]
of the nerve of $\mathrm{w}\mathbf{DSp}_\bullet(S)$, starting with
maps from $\one$ to $\fib(\E^+_0\smash \E^-_n)$ and using the duality
datum $\E^+_0\smash \E^-_n \rightarrow \fib(\one)$ obtained via compositions
instead. 
The map to $\fib(\fib(\one))$
described above for $D(\E^+,\E^-,e^-)$ extends to $D(\underline{\E})$ for
$\underline{\E}$ such an $n$-simplex.
A map $\alpha\colon [m]\rightarrow {[n]}$ in $\Delta$ induces a
map $D(\underline{\E})\rightarrow D(\alpha^\ast(\underline{\E}))$. 
Consider the induced map of bisimplicial sets
\[ \coprod_{\underline{\E}\in \mathrm{w}\mathbf{DSp}_\bullet(S)} D(\underline{\E})
 \rightarrow \coprod_{\underline{\E}\in \mathrm{w}\mathbf{DSp}_\bullet(S)} \{\underline{\E}\}.\]
Since $D(\underline{\E})$ is contractible, this map is a weak equivalence 
by the realization lemma. Moreover, it maps
to $\fib(\fib\one)(S)$, and this map is natural in $S$.
It may be regarded as the zeroth
level of a map of $S^1$-spectra 
\[ \AAA_T^\du(S) \rightarrow \fib(\fib\one)(S), \]
which is natural in $S$. Instead of providing a spectrum-level map
by explicit constructions similar to those appearing in the proof 
of \cite[Theorem 0.1]{May:traces}, however,
\cite[Theorem 6.5]{hss.higher-traces}, which in turn refers
to~\cite{toen-vezzosi.chern}, will be invoked.

\begin{theorem}\label{thm:s-trace}
  Let $S$ be a Noetherian finite-dimensional base scheme.
  There exists a multiplicative map $\AAA^\du_T(S) \rightarrow \fib(\one)(S)$
  of symmetric ring $S^1$-spectra which is natural in $S$.
  On path components it sends the class given by the
  dualizable $T$-spectrum $\E$ to the Euler characteristic
  of $E$, considered as an endomorphism $\chi(\E)\colon \one\rightarrow \fib(\one)$.
\end{theorem}

\begin{proof}
  The category $\mathbf{SymSp}^\du_T(S)$ of cofibrant dualizable 
  symmetric $T$-spectra over $S$ gives rise to a small, stable, 
  idempotent-complete, rigid symmetric monoidal $\infty$-category
  $\mathcal{E}(S)$,
  naturally in $S$. Hence it fits into the general framework
  of \cite{hss.higher-traces}. More specifically, 
  \cite[Theorem 6.5 and Remark 6.6]{hss.higher-traces} apply to give a
  map of symmetric ring spectra
  \[ \AAA(\mathcal{E}(S)) \rightarrow \fib(\one)(S), \]
  whose domain is the $\infty$-categorical 
  $K$-theory of $\mathcal{E}(S)$ as introduced in
  \cite[Remark 1.2.2.5]{lurie.ha}, and whose target is regarded as
  the endomorphism $S^1$-spectrum of the unit in $\mathcal{E}(S)$.
  The Waldhausen $K$-theory $\AAA^\du_T(S)$ of $\mathbf{SymSp}^\du(S)$
  maps via a natural weak equivalence to $\AAA(\mathcal{E}(S))$ \cite[Theorem 7.8]{bgt}, 
  and this so as a symmetric
  ring spectrum \cite[Proposition 5.8]{bgt-2}. 
  This provides the desired multiplicative map of $S^1$-spectra.
  For the statement regarding path components, see 
  \cite[Remark 6.6]{hss.higher-traces}.
\end{proof}

It would be very interesting to relate the homotopy fiber
of the trace from Theorem~\ref{thm:s-trace}
with geometrical data.

\begin{corollary}\label{cor:trace-field-zero}
  Let $F$ be a field of characteristic zero. There
  exists a multiplicative trace map
  \[ \AAA^\hf_T(F) \rightarrow \fib(\one)(F), \]
  which induces a ring homomorphism
  \[ \pi_0\AAA^\hf_T(F) \rightarrow \pi_0\fib(\one)(F)\cong \mathrm{GW}(F) \]
  to the Grothendieck-Witt ring of $F$.
\end{corollary}

\begin{proof}
  This follows from Theorem~\ref{thm:s-trace} and
  Proposition~\ref{prop:finite-dualizable}. The statement
  on path components involves Morel's theorem \cite[Cor.~1.25]{morel.field} 
  computing the path components of the endomorphism
  spectrum of the sphere $T$-spectrum. More precisely, the
  global sections of $\fib(\one)$ coincide with the infinite
  $T$-loop space (or spectrum) associated with a fibrant replacement of
  the $T$-suspension spectrum of the zero sphere $S^0_F$ over $F$.
  Hence by construction its path components form the endomorphism
  ring of $\one\in \SH(F)$, which Morel computed to be naturally
  isomorphic to the Grothendieck-Witt ring of $F$. 
\end{proof}

The composition of the motivic measure
\[ K_0(\Var_F)\rightarrow \pi_0\AAA^\hf_T(F)=\pi_0\AAA^\du_T(F) \]
induced by Theorem~\ref{thm:a-motivic-measure} and
the ring homomorphism from Corollary~\ref{cor:trace-field-zero}
provides a motivic measure to the Grothendieck-Witt ring on $F$.
For formal reasons -- see, for example,~\cite{May:traces} --
it extends the categorical Euler characteristic
$K_0(\SH^\du(F))\to [\one,\one]_{\SH(F)}\cong \mathrm{GW}(F)$ used
in refined enumerative geometry by Levine, Wickelgren, and others \cite{levine.mhreg}.
As another application of Theorem~\ref{thm:s-trace}, an interesting
$\Aff^1$-homotopy type results.

\begin{corollary}\label{cor:dualizable-nontrivial}
  Let $S$ be a Noetherian finite-dimensional base scheme.
  The fibrant replacement of the presheaf $\AAA^\du_T\in \SymSp(S)$
  factors the unit map
  \[ S^0 \to \Omega^\infty_T\Sigma^\infty_TS^0  \to \Omega_T^\infty \fib(\one)\in \SymSp(S).\]
  In particular, the $\Aff^1$-homotopy type associated with $\AAA^\du_T$ is
  nontrivial, and the global sections of  its $\Aff^1$-path components $\pi_0^{\Aff^1}\AAA^\du_T$  factor the motivic measure
  \[ K_0(\Var_F) \to \pi_0^{\Aff^1}\AAA^\du_T(F) \to \mathrm{GW}(F) \]
  from Corollary~\ref{cor:trace-field-zero} in case $S=\mathrm{Spec}(F)$ is the spectrum
  of a field $F$ of characteristic zero.
\end{corollary}

\begin{proof}
  The proof of Corollary~\ref{cor:trace-field-zero} already mentioned that
  the global sections of $\fib(\one)$ over any Noetherian finite-dimensional
  base scheme such as $X\in \Sm_S$ is the infinite $T$-loop space of
  the sphere spectrum $\one_X\in \SH(X)$, viewed as a symmetric $S^1$-spectrum.
  Hence the naturality in Theorem~\ref{thm:s-trace} provides a map
  \[ \AAA^\du_T \to \Omega_T^\infty \fib(\one)\in \SymSp(S), \]
  which sectionwise is a multiplicative map of symmetric ring spectra.
  In particular, this map is naturally compatible with the unit.
  Since its target $\Omega_T^\infty \fib(\one)$ is fibrant
  by construction, the map factors over a fibrant replacement
  of $\AAA^\du_T$, giving
  rise to the claimed factorization
  of the unit map. The remaining statement then follows from
  Corollary~\ref{cor:trace-field-zero}.
\end{proof}

\bigskip

{\bf Acknowledgements}: 
Friedhelm Waldhausen suggested to consider algebraic $K$-theory
of motivic spaces to me a long time ago. I thank him for his
support during many years. Basically all results except the ones listed in
Section~\ref{sec:trace} were sketched during a short
presentation at the Union College Mathematics Conference 2003.
I thank Lars Hesselholt and Marc Levine for discussions and encouragement
after further talks I gave on the subject in 
2011 and 2015, respectively. Finally I thank Bj{\o}rn Ian Dundas and
an anonymous referee for very helpful comments.

%\bibliographystyle{alpha}
%\bibliography{atheory}

\end{document}